\documentclass[reqno,final]{amsart}
\usepackage{color} %Color definition
\usepackage{hyperref} %Internal and external links
\usepackage{graphicx} %Graphics inclusion

\usepackage{pstricks}
\usepackage{amssymb}
\usepackage[utf8]{inputenc}
\usepackage[T1]{fontenc}
\usepackage{lmodern}
\usepackage{amsmath,amssymb}
\usepackage{bm}
\usepackage{amsthm}
\usepackage{url}
\usepackage{graphicx}
\usepackage{tikz}
\usepackage{float}
\usepackage{dsfont}
\usepackage{comment}
\usepackage[english]{babel}
\usepackage{stmaryrd}
\usepackage{mathrsfs}
\theoremstyle{plain}
\newtheorem{theorem}{Theorem}[section]                                          
\newtheorem{proposition}[theorem]{Proposition}                          
\newtheorem{lemma}[theorem]{Lemma}
\newtheorem{corollary}[theorem]{Corollary}

\theoremstyle{definition}
\newtheorem{definition}[theorem]{Definition}
\theoremstyle{remark}

\makeatletter \@addtoreset{equation}{section} \makeatother

\newcommand{\dro}{\partial_t}  % dr comme "d rond" pour les dérivées partielles, plus court...
\newcommand{\E}{\mathbb{E}} % marre d'écrire R en galérant!

\newcommand{\PP}{\mathbb{P}}
\newcommand{\Z}{\mathbb{Z}}
\newcommand{\N}{\mathbb{N}}

\newcommand{\R}{\mathbb{R}}

\title{Spatialized evolutionary prisoner's dilemma: homogenization and propagation of chaos}

\author{Sylvain GIBAUD}

\email{sylvain.gibaud@math.univ-toulouse.fr}


\begin{document}

\begin{abstract}
 Epstein \cite{epstein1998zones} introduced an agent-based model called Demographic Prisoner's dilemma. He shows, via simulations, that cooperation in this spatial evolutionary repeated game can be sustained.\\ In order to do proves, we put on this model a particle system framework in order to prove homogenization and propagation of chaos results. Firstly we prove the convergence of the spatial model to a random matching model, using homogenization techniques. Then we prove the convergence of the random matching model to a mean field model, using propagation of chaos techniques.
\end{abstract}

\maketitle

	\section{Introduction}
	Playing particle systems as a modeling tool in the evolution literature have been introduced by Smith and Price \cite{smith1973lhe}. Since, many ecological models have used playing particle system's approach (see a list of examples in the book of Hartl and Clark \cite{hartl1997principles}). From the math point of view, there are four main approaches to build models on evolution, according to Durrett and Levin \cite{durrett1994importance}.
	\begin{itemize}
		\setlength\itemsep{0.1em}
		\item In the mean field approach, the system is homogeneous in space and is described through the densities of particles in each state. The evolution of those densities are described by ordinary differential equations. 
		\item In the patch model approach, discrete individuals are grouped into patches without additional spatial structure. We will be interested in a modified version of the second one, called the random matching approach introduced by Gilboa \cite{gilboa1992model}.  In this model there is only one patch, all the particles are on this patch and a couple of particles are drawn uniformly on all couples to interact. This approach is used for example by  Ellison \cite{ellison1994cooperation}. The system is described through a Markov process.
		\item In the reaction-diffusion approach, the individuals are infinitesimal and distributed in space. The system is described through the densities of particles in each state and each position. Those densities are described by partial differential equations.
		\item In the interacting particle system approach, the individuals are discrete and space is treated explicitly. The system is described through a Markov process.
	\end{itemize}
	
	The particles system model we study is the Demographic Prisoner's Dilemma, introduced by Epstein \cite{epstein1998zones}. Let us introduce a modified (meaning without birth) version of Epstein model as follow:
	\begin{enumerate}
		\setlength\itemsep{0.2em}
		\item Let $(\Z / m \Z)^2$ be a fixed torus with $m \in \N^*$.
		\item $N$ particles are placed on the torus ($N \in \N^*$). For a particle $i \in \lbrace 1,\dots, N \rbrace$, let us denote $X_i(t)$ its position at time $t \in \R_+$. 
		\item The $N$ particles move according to continuous time independent symmetric simple random walks. Thus we associate to each particle $i$ a Poisson process with parameter $d\in \N^*$. When the Poisson process realizes, $X_i(t)$ jumps to one of its nearest neighbors with equal probability. Then two (or more) particles can be on the same site at the same time, as shown in the following example of evolution.
			\begin{center}
					\includegraphics[scale=0.25]{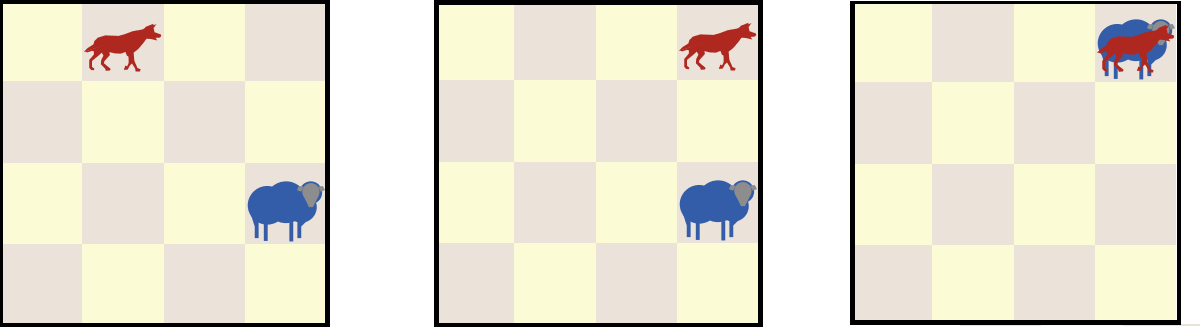}
				\end{center}
	
		\item Each particle $i \in \lbrace 1,\dots N \rbrace$ has a wealth at time $t \geq 0$ denoted $Y_i(t)\in \R_+$. If $Y_i$ reaches non positive number, it goes to 0 and stays there. It means that the particle dies.
		\item   The wealth changes through games. The game is the prisoner's dilemma:
		\begin{itemize}
			\setlength\itemsep{0.2em}
			\item  The players have two actions to \textit{Cooperate} or to \textit{Defect}.
			\item If both \textit{Cooperate} they get a Reward $\mathbf{R}$. But if one of the two {defects}, the \textit{Defector} gets a Temptation payoff $\mathbf{T}$, and this payoff is bigger than the Reward. If a \textit{Cooperator} is being \textit{Defected} instead of getting a reward he gets a Sucker payoff $\mathbf{-S}$. If both players \textit{Defect} (let us call them Player 1 and Player 2)  the nature flips a coin, if it is head Player 1 gets a Punishment payoff of $-2 \mathbf{P}$ and Player 2 gets a payoff of 0, if it is tail Player 2 gets a Punishment payoff of $-2 \mathbf{P}$ and Player 1 gets a payoff of 0.
			\item Then the payoff satisfies $\mathbf{T}>\mathbf{R}>0>$ and $\mathbf{S}>\mathbf{P}>0$. This is summarize in the following payoff matrices where at each game a coin is tossed determining which payoff matrix is used (action Top and action Left are \textit{Cooperate}, action Bottom and action Right are \textit{Defect}).\\
				\begin{eqnarray}\label{Intro payoffs 1}
				\left(\begin{array}{cc}
				(\mathbf{R},\mathbf{R}) & (-\mathbf{S},\mathbf{T}) \\ (\mathbf{T},-\mathbf{S}) & {(-2\mathbf{P},0) }  
				\end{array}\right),
				\end{eqnarray}
				
				\begin{eqnarray}\label{Intro payoffs 2}
				\left(\begin{array}{cc}
				(\mathbf{R},\mathbf{R}) & (-\mathbf{S},\mathbf{T}) \\ (\mathbf{T},-\mathbf{S}) & {(0,-2\mathbf{P}) }
				\end{array}\right).
				\end{eqnarray}
			\item A point to notice is that two players cannot lose wealth simultaneously.
			\item The unique strong Nash equilibrium is \textit{(Defect,Defect)}.
			\item To determine the action played by the individuals we place ourself in the framework where: "Each individual plays only one action, either he will \textit{Defect} every time either he will \textit{Cooperate} every time". One intuition behind it is: particles represent bacterias, and their actions is embedded in their DNA so they cannot change their action. Then a particle has a fixed action (\textit{Cooperate} or \textit{Defect}). There are two kinds of particles: the ones who always \textit{Cooperate} and the ones who always \textit{Defect}. We modelize this by the following: Particle $i$ plays his action according to a parameter $z_i  \in \lbrace 0,1 \rbrace$, called strategy which is constant over time. Particle $i$ plays \textit{Cooperate} if $z_i = 0$ and plays \textit{Defect} if $z_i = 1$.
			\item At the end of each game (between for example player $i$ and $j$) $Y_i$ and $Y_j$ are updated, adding the payoff of the game played by $i$ and $j$.
			\item To make the games happen, each couple of individuals $(i,j)$ is given a Poisson process independent of everything of parameter $v$. When this Poisson process realizes (for example at a time $t$), if the individuals are on the same site (\emph{i.e.} if $X^i(t)=X^j(t)$) and if their wealths are positive (\emph{i.e.} if $Y_i(t)>0$ and $Y_j(t)>0$) the individuals play together. Otherwise nothing happens.
			\item If the wealth of a particle reaches non positive then the particle stops playing with the other players forever.
		\end{itemize} 
		\item We also make an important assumption (useful for the propagation of chaos):  we assume that the particles are indistinguishable, \emph{i.e.} the distribution of $(X,{Y},z)$ is exchangeable \emph{i.e.} for every permutation $\sigma$ of $\lbrace 1, \dots, N \rbrace$, \[
			\begin{array}{l}
	\mathcal{L}\left((X_1,{Y_1},z_1), \dots, (X_N,{Y_N},z_N)\right) = \\
	\mathcal{L}\left((X_{\sigma(1)},{Y_{\sigma(1)}},z_{\sigma(1)}),\dots, (X_{\sigma(N)},{Y_{\sigma(N)}},z_{\sigma(N)})\right).
		\end{array}
		 	\]

	\end{enumerate}
	
	We have that $(X(t),Y(t),z)_t$ is a Continuous Time Markov Chain, we denote by $(\mathcal{F}^d_t)_t$ its filtration such that $(X(t),Y(t),z)_t$ is adapted to $(\mathcal{F}^d_t)_t$.\\

	The main contribution of this article is to prove the convergence as $d\to +\infty$ of the spatial model to a random matching model and then as $N\to +\infty$ to the mean field model. To do that, we introduce a probabilistic framework to the Epstein model using interacting particle system. We prove that the induced Markov process converges in law when $d\to +\infty$ to a Markov process induced by a random matching model. Then we show that when $N \to +\infty$ the latter Markov process converges in law to a non-linear Markov process induced by a mean field model.\\
	Epstein's model was studied by Dorofeenko and Shorish \cite{dorofeenko2002dynamical}. Assuming statistical independence they proved the convergence to a reaction-diffusion process. Evolutionary models through playing particles have been studied in the case where the evolution is driven by birth-death process by Champagnat and al. \cite{champagnat2004etude,champagnat2006unifying,champagnat2007invasion}. \\
	When we lose the sight of the evolution and of the games, Bordenave and al. \cite{bordenave2007particle} proved the convergence of an interacting particles system in a rapid varying environment to a mean field system.  The Mean Field Games have been introduced by Lasry and Lions \cite{lasry2007mean} in a context where there is neither space, nor evolution. They proved the convergence of Nash equilibria when the payoffs depends on the empirical mean and when the number of players goes to infinity.\\
	
	Let us state the main results of this article.\\
	
	We denote by $(e_1,\dots,e_N)$ the canonical basis of $\R^N$, $(e_1^c,\dots,e_N^c)$ ($c \in \lbrace 1,2\rbrace$) the canonical basis of $(\R^2)^N$ and we set $m \in \N^*,\ d\in \N^*,\ v\in \R_+$. We also identify $\left((\Z/m\Z)^2 \times \R_+ \times \lbrace 0,1 \rbrace\right)^N$ to $((\Z/m\Z)^2)^N \times \R_+^N \times \lbrace 0,1 \rbrace^N$, then we can write $(x,y,z)\in\left((\Z/m\Z)^2 \times \R_+ \times \lbrace 0,1 \rbrace\right)^N$ with $x\in ((\Z/m\Z)^2)^N,  y\in \R_+^N,z\in \lbrace0,1\rbrace^N$.  \\
	We describe the three approaches, we are interested in, by infinitesimal generators.
	\vspace{20pt}
		
	\textbf{Particle system approach}\\

	The generator of the Markov process presented earlier is: for all continuous bounded functions $f$ of $\left((\Z/m\Z)^2 \times \R_+ \times \lbrace 0,1 \rbrace\right)^N$,
	\begin{eqnarray}\label{A Intro}
	\mathcal{A} f(x,y,z) = \mathcal{A}_d f(x,y,z) + v\,\mathcal{A}_g f(x,y,z)
	\end{eqnarray}
	where $\mathcal{A}_d$ is the generator representing the movement of the particles. The particles move following independent random walks so $\mathcal{A}_d$ is defined by: for all continuous bounded functions $f$ of $\left((\Z/m\Z)^2 \times \R_+ \times \lbrace 0,1 \rbrace\right)^N$: 
	\[
	\mathcal{A}_d f(x,y,z) = \sum\limits_{i=1}^{N} \sum\limits_{c=1}^2 \sum\limits_{\epsilon = \pm 1} \frac{d}{2 \times 2} [f(x + \epsilon e_i^c,y,z) - f(x,y,z)].
	\]
	The generator $\mathcal{A}_g$ is the generator representing the evolution of the wealth of particles through games. Let us give firstly $\mathcal{A}_g$, then the explanation on how it works. We have for all continuous bounded functions $f$ of $((\Z/m\Z)^2 \times \R_+ \times \lbrace 0,1 \rbrace)^N$: $\mathcal{A}_g f(x,y,z)$ is equal to
	\[
	\sum\limits_{\substack{(i,j) \in \lbrace 1,\dots N \rbrace^2\\
			i \neq j}}  \frac{1}{2} \mathds{1}_{x_i = x_j} \mathds{1}_{\substack{y_i > 0\\y_j >0}} \left[\begin{array}{l}\ \mathds{1}_{\substack{z_i = 0 \\ z_j =0}} (f(x,y + \mathbf{R} e_i + \mathbf{R} e_j,z) - f(x,y,z)) + \smallskip\\
	2\mathds{1}_{\substack{z_i = 1\\z_j = 0}} (f(x,y + \mathbf{T} e_i - \mathbf{S} e_j,z) - f(x,y,z)) +  \qquad\smallskip \\
	\ \mathds{1}_{\substack{z_i = 1\\z_j =1}} \left(\begin{array}{l}\ \frac{1}{2}(f(x,y -2\mathbf{P} e_j,z) - f(x,y,z))  \\
	+ \frac{1}{2}(f(x,y -2\mathbf{P} e_i,z) - f(x,y,z))\end{array}\right) \ 
	
	\end{array} \right]
	\]
	Let us explain a bit this generator:
	\begin{itemize}
		\item $\mathds{1}_{x_i = x_j}$ represents the spatial structure meaning that a game makes the wealth change only if the two particles (with positions $x_i$ and $x_j$) are on the same site.
		\item $\mathds{1}_{\substack{y_i > 0\\y_j >0}}$ checks if both particles are alive (\emph{i.e.} there wealth ($y_i$ and $y_j$) are positive)
		\item the big bracket represents the change of wealth through the game of the two particles. The term $\mathds{1}_{\substack{z_i=1\\z_j=0}}$ looks at the strategies of the players and choose the right payoff to give to the particles. For example if particle 1 and 2 are playing (both are alive and on the same site) if particle 1 cooperates and 2 defects then particle 1 gets a payoff of $-\mathbf{S}$ and particle 2 gets $\mathbf{T}$, the term in the generator representing this type of transition is: 
		\[
		\mathds{1}_{\substack{z_1=0\\z_2=1}} [f(x,y -\mathbf{S}e_1 + \mathbf{T}e_2,z) - f(x,y,z)].
		\]
	\end{itemize}
	\vspace{20pt}
	\textbf{Random matching approach}\\

	The random matching generator we use (obtained removing the moving part and re-normalizing by the size of the torus) is:  
		for all continuous bounded functions $f$ of $(\R_+ \times \lbrace 0,1 \rbrace)^N$:
		\begin{eqnarray} \label{RM Intro}
		\overline{\mathcal{A}}f(y,z) = \frac{v}{m^2} \overline{\mathcal{A}_g} f (y,z)
		\end{eqnarray}

		where $\overline{\mathcal{A}_g} f (y,z)$ is equal to
		\begin{eqnarray}\label{Géné RM Intro}
		  \sum\limits_{\substack{(i,j) \in \lbrace 1,\dots N \rbrace^2\\
				i \neq j}} \frac{1}{2}  \mathds{1}_{\substack{y_i > 0\\y_j >0}} \left[\begin{array}{l}\ \mathds{1}_{\substack{z_i = 0\\ z_j =0}} (f(y + \mathbf{R} e_i + \mathbf{R} e_j,z) - f(y,z)) + \\
		2\mathds{1}_{\substack{z_i = 1\\z_j = 0}} (f(y + \mathbf{T} e_i - \mathbf{S} e_j,z) - f(y,z)) +  \qquad\smallskip \\
		\ \mathds{1}_{\substack{z_i = 1\\z_j =1}} \left(\begin{array}{l}\ \frac{1}{2}(f(y -2\mathbf{P} e_i,z) - f(y,z))  \\
		+ \frac{1}{2}(f(y -2\mathbf{P} e_j,z) - f(y,z))\end{array}\right) \ 
		
		\end{array} \hspace{-10pt}\right]
		\end{eqnarray}
		This generator describe the whole system. It is the generator of a Continuous Time Markov Chain. \\
		
		But we want to look at only one particle of this system. So we use two tools: a generator depending on a measure to code the interactions, and a measure to code the information of the particles in the system.\\
		
		The usual tool to look at the particles is the empirical measure.
		With $(y,z)=((y_1,z_1),\dots,(y_N,z_N)) \in (\R_+,\lbrace 0,1 \rbrace)^N$ we set:
		\[
		m^{N}(y,z) = \frac{1}{N} \sum_{j=1}^{N} \delta_{(y_j,z_j)}.
		\]
		To describe the evolution of one particle $i\in\lbrace1,\dots,N \rbrace$, the usual way is taking in $\overline{\mathcal{A}}f$ a function $f$ constant on \{$(y_j,z_j)$\} for $j\neq i$. This operator is for all $f$ continuous bounded from $\R_+\times \lbrace 0,1 \rbrace$ to $\R$, for all $(y,z) \in (\R_+,\lbrace 0,1 \rbrace)^N$
		
	\begin{eqnarray}\label{Gene une particle dans random matching Intro}
	\frac{Nv}{m^2} \left(L_{m^{N}(y,z)}f(y_i,z_i) \right)  + h(y_i,z_i)
	\end{eqnarray}
		
		where 
		\begin{itemize}
		\setlength\itemsep{0.2em}
		\item $L$ is for $f$ a continuous bounded from $\R_+ \times \lbrace 0,1 \rbrace$ to $\R$ and $(y,z) \in \R_+ \times \lbrace0,1\rbrace$
		\begin{eqnarray}\label{Interaction Seul Intro}
				L_{m} f(y,z) =  \frac{1}{2} \mathds{1}_{y>0}\left[\begin{array}{l}
				\ \ \mathds{1}_{z=0}\, m((0,+\infty)\times \lbrace 0 \rbrace) (f(y + \mathbf{R},z) - f(y,z))\smallskip\\
				+\mathds{1}_{z=0}\, m((0,+\infty)\times \lbrace 1 \rbrace) (f(y - \mathbf{S},z) - f(y,z)) \smallskip\\
				+\mathds{1}_{z=1}\, m((0,+\infty)\times \lbrace 0 \rbrace) (f(y + \mathbf{T},z) - f(y,z)) \smallskip\\
				+\mathds{1}_{z=1}\, m((0,+\infty)\times \lbrace 1 \rbrace) \frac{1}{2}(f(y - 2\mathbf{P},z) - f(y,z))
				\end{array}\right].
		\end{eqnarray}.
		$L$ is a generator depending in a measure. The heuristic behind it is to look at the evolution of one particle in an environment ruled by this measure. For $m$ a measure of $\R_+ \times \lbrace 0,1\rbrace$ the generator 
		\item $h(y_i,z_i) = \frac{v}{2 m^2} \left(\mathds{1}_{z_i=z_j=0} [f(y_i + R) - f(y_i)] + \frac{1}{2}\mathds{1}_{z_i=z_j=1}[f(y_i - 2\mathbf{P}) - f(y_i)]\right)$ is a bounded compensating term. It compensates the fact that in $L_{m^N(y,z)}$ particle $i$ can play with herself. We have that $|h(y_i,z_i)| \leq 2\frac{v}{m^2}\Vert f \Vert_\infty$.
		\end{itemize}

		Then the generator described in $(\ref{Gene une particle dans random matching Intro})$ does not generate a Markov process (because $m^{N}$ does not depend only on $(y_i,z_i)$).\\
		\vspace{20pt}

		\textbf{Mean Field approach}\\
		
		Before introducing the generator, let us set some notations. For $\mu$ a probability measure on the space of sample path right-continuous left-limited (also called càd-làg) of $\R_+\times \lbrace0,1\rbrace$ $D(\R_+,\R_+\times \lbrace0,1 \rbrace)$ we set: for all $S$ Borel sets of $\R_+ \times \lbrace 0,1 \rbrace$
		\[
		\mu_t(S) = \mu(X_t \in S)
		\] 
		The stochastic process describing the mean field approach is a non-linear Markov process. That means that the distribution $\mu$ of the process is described by a generator also depending on $\mu$. For that purpose we call this kind of generator non-linear generator. \\
		 The distribution of its sample path is $\mu$ and the generator of $\mu$ is: for all $(y,z) \in \R_+ \times \lbrace0,1\rbrace$
		\[
		\frac{v}{m^2} L_{\mu_t}f(y,z)
		\]
		\vspace{3pt}\\
		Now that we have set the previous generators, we can state our results.\\
        Since the distribution of $(X,Y,z)$ depends on $d$, when necessary we make this dependence visible denoting for example $(X^d,Y^d,z)$.
	\begin{theorem}
		Let $(X^d(t),Y^d(t),z)_t$ be a stochastic process with some initial distribution $\nu$ of $((\Z/m\Z)^2\times (\R_+ \times \lbrace 0,1 \rbrace))^N$ with second marginal denoted $\nu_y$ and generator $\mathcal{A}$ defined in $(\ref{A Intro})$.\\	
		Then $(Y^d,z)_d$ converges in distribution (as a sequence of stochastic processes) to the Markov process $(\overline{Y},z)$ with infinitesimal generator $\overline{\mathcal{A}}$ (defined in (\ref{RM Intro})) of domain the space of continuous bounded functions of $\R_+ \times \lbrace 0,1 \rbrace$ and initial distribution $\nu_y$.
	\end{theorem}
	
	This theorem is a homogenization theorem. It mainly uses the averaging theorem of Kurtz \cite{kurtz1992averaging}. It was used in many examples, we only quote the closest to our problem which is the article of Bordenave and al. \cite{bordenave2007particle}.\\
	The following theorem is a propagation of chaos result. To have it, we have to "slow" the time dividing the generator by $N$. Let this generator $\overline{\mathcal{A}}$ be defined by: for all continuous bounded functions $f$ from $\R_+ \times \lbrace 0,1 \rbrace$ to $\R$: $\forall (y,z)\in (\R_+ \times \lbrace0,1 \rbrace)^N$
	
	\[
	\overline{\mathcal{A}}f(y,z) = \frac{1}{N}\frac{v}{2 \,m^2} \overline{\mathcal{A}_g} f (y,z)
	\]
	with $\overline{\mathcal{A}_g}$ defined in (\ref{Géné RM Intro}).
	\begin{theorem}
		Let $(\hat{Y}^N(t),z)_t = ((\hat{Y}^N_1(t),z_1),\dots (\hat{Y}^N_N(t),z_N))_t$ be a stochastic process with generator $\overline{\mathcal{A}}$ and with initial measure $\nu^{\otimes N}$.\\
		Then when $N \to +\infty$, for all $i\in \N$, $\mathcal{L}((\hat{Y}^N_i(t),z_i)_t)$ converges to $\mu$ with initial condition $\nu$ and with non-linear Markov generators $(L_{\mu_t})_t$ ($L$ defined in (\ref{Interaction Seul Intro})).
	\end{theorem}

	\begin{corollary}
		The previous $\mu$ can be described by the following evolution equations.\\
				Let $(Y(t),z)_t$ be a process of law $\mu$ starting with the initial distribution $\nu$. We have: for all $y\in \R^*_+$
				\begin{eqnarray} \label{Evol B Intro}
				\begin{array}{rcl}
				\dro \PP(Y(t) = y,z =1) & = & \ \ \PP(Y(t) = y - \mathbf{R},z= 1)\, \PP(Y(t)>0,z=1)\mathds{1}_{y>\mathbf{R}}\smallskip\\
				& & + \PP(Y(t) = y + \mathbf{S}, z=1)\, \PP(Y(t) >0, z=0)\smallskip\\
				& & - \PP(Y(t)=y,z = 1)\,\PP(Y(t)>0)
				\end{array}
				\end{eqnarray}

				\begin{eqnarray} \label{Evol R Intro}
				\begin{array}{rcl}
				\dro \PP(Y(t) = y,z =0) & = &\ \ \ \  \PP(Y(t) = y - \mathbf{T},z= 0)\, \PP(Y(t)>0,z=1)\mathds{1}_{y >\mathbf{T}}\smallskip\\
				& & + \frac{1}{2}\PP(Y(t) = y + 2\mathbf{P}, z=0) \,\PP(Y(t) >0, z=0)\smallskip\\
				& & + \frac{1}{2}\PP(Y(t) = y, z=0)\, \PP(Y(t) >0, z=0)\smallskip\\
				& & - \ \ \PP(Y(t)=y,z = 0)\,\PP(Y(t)>0)
				\end{array}
				\end{eqnarray}
				
	\end{corollary}

	We see that when $N \to +\infty$, the system is described by ordinary differential equations. Moreover with this theorem, there exists a probability measure $\mu$ such that the random matching model is $\mu$-chaotic. This means in particular that for all couple of particles $(i,j)$ the joint law of $(i,j)$ is when $N \to +\infty$: $\mu \otimes \mu$. The law $\mu$ is described by the evolution equations (\ref{Evol B Intro}),(\ref{Evol R Intro}).\\
	This theorem is a propagation of chaos result. One important reference about propagation of chaos is the Saint Flour course of Sznitman \cite{sznitman1991topics}. This topic was also studied by  Graham and al. \cite{graham2000chaoticity,graham1997stochastic}, and used Bordenave and al. \cite{bordenave2007particle}. Also Graham \cite{graham1995homogenization} make homogenization and propagation of chaos on just moving particles on a bounded domain with absorption and desorption arriving at a fast scale.\\
	In the first part, we introduce a modified version of Epstein model and give a probabilistic framework (together with some reminders on Markov processes) to this new model. In the second part, we state and prove the homogenization theorem. In the third part we give some reminders on propagation of chaos and prove the propagation of chaos theorem. In the last part we extend the previous results to a more general context (where the game is not necessarily a prisoner's dilemma game, there can be more than two strategies which are not necessarily pure and in a more general graph).

\section{The Model}\label{Model}

Let us state the first notations.\\

\textbf{Notations}\\
	Let $(\Omega,\mathcal{T},\PP)$ be a probability space.
	Let $(e_1,\dots,e_N)$ be the canonical basis of $\R^N$ (with $N\in \N^*$).
	Let $(e_1^1,e_1^2,e_2^1,\dots, e_N^1,e_N^2)$ be the canonical basis of $(\R^2)^N$.\\
	For $E$ a complete separable metric space, $\mathcal{L}(X)$ is the distribution of the $E$-valued random variable $X$, \\
    Let us denote $\overline{C}(E)$ the set of continuous bounded functions from $E$ to $\R$.
	We denote by $\mathcal{P}(E)$ the probability space on $E$.\\
	If $E_1,E_2$ are two complete separable spaces such that $E=E_1 \times E_2$, we identify the functions in $\overline{C}(E)$ constant on the first variable with the functions in $\overline{C}(E_2)$.\\
\newline

Before introducing the model in details let us give some reminders about Markov processes.\\
One good way to study Markov processes is via their distribution. For a complete separable state space $F$, a process $X=(X(t))_{t\geq0}$ has value in the canonical space.
\begin{definition}[Canonical space]
The canonical space $D(\R_+,F)$ is the space of right continuous functions from $[0,+\infty)$ to $F$ with left limits, endowed with the Skohorod topology associated to its usual metric. With this metric, $D(\R_+,F)$ is complete and separable (see the book of Ethier and Kurtz \cite{EK} for more details).\\

We denote by $(\mathcal{T}_t)_t$ the natural filtration associated to $D(\R_+,F)$.
We define the canonical process $X := (X(t))_{t\geq0}$ by:\[
\forall t\geq 0, \qquad	X(t): \omega \in D(\R_+,F) \mapsto \omega(t) \in F
	\]
\end{definition}

To define a Markov process in a discrete countable state space $F$ (which is the case here), we need an initial measure $\nu$ on $F$ and a matrix $(\mathcal{A}(x,y))_{x,y\in F}$ with $\mathcal{A}(x,x) = - \sum\limits_{\substack{y\in F\\ y\neq x}} \mathcal{A}(x,y)$. For $x\neq y$ $\mathcal{A}(x,y)$ gives the rates of transition of the future process from $x$ to $y$. With this matrix we define the generator (which is one of the main tool in the study of Markov processes) of the following Markov process such that: for all $f$ bounded and measurable (for the subsets of $F$) from $F$ to $\R$ and for all $x\in F$:
\[
\mathcal{A}f(x) = \sum\limits_{y \in F} \mathcal{A}(x,y)[f(y)-f(x)]
\]
The generator is a bounded linear operator on the bounded functions of $F$. To make the difference between the rate matrix and the operator generator. We will always denoting the matrix with two entries (for example $\mathcal{A}(.,.)$) and the generator with no entry (for example $\mathcal{A}$).\\
Conversely giving a generator we can construct the rate matrix looking at the rates in the expression of the generator.

\begin{definition}\label{Construct Probabiliste Process Markov }

Let $\mathcal{A}$ be a and an initial measure $\nu$ we can construct a Markov process $X$ as follow:\begin{enumerate}
\item Let $(Y(n))_n$ be a Markov chain in $F$ with initial distribution $\nu$ and with transition matrix $\left(\frac{\mathcal{A}(x,y)}  {|\mathcal{A}(x,x)|}\right)_{x,y \in F}$
\item Let $\Delta_0,\Delta_1,\dots, $ be independent and exponentially distributed with parameter 1 (and independent of $Y(.)$) random variables.
\item We define the Markov process $(X(t))_t$ in $F$ with initial distribution $\nu$ and generator $\mathcal{A}$ by:

\begin{eqnarray}
X(t) = \left\lbrace \begin{array}{ll}
Y(0), & 0 \leq t < \frac{\Delta_0}{|\mathcal{A}(Y(0),Y(0))|} \medskip\\
Y(k), & \sum\limits_{j=0}^{k-1}\frac{\Delta_j}{|\mathcal{A}(Y(j),Y(j))|} \leq t < \sum\limits_{j=0}^{k} \frac{\Delta_j}{|\mathcal{A}(Y(j),Y(j))|}
\end{array}\right.
\end{eqnarray}
\end{enumerate}
Note that we allow $\mathcal{A}(x,x) = 0$ taking $\Delta/0 = \infty$.
\end{definition}

Now that Continuous Time Markov Chains are properly defined. \\
Let us link them to one of the key tool in probabilities: Martingales. Since martingales are easier to deal with than Markov processes (speacking of convergences for example) the following proposition (Proposition 1.7 of Ethier Kurtz \cite{EK}) is really useful. 

\begin{proposition}
Let $(X(t))_t$ a càd-làg Markov process adapted to the filtration $(\mathcal{F}_t)_t$ with generator $\mathcal{A}$. We have: for all continuous bounded functions $f$ from $F$ to $\R$
\[
M^f(t) := f(X(t)) - \int\limits_{0}^{t} \mathcal{A}f(X(s)) \text{d}s
\] is a $\mathcal{F}_t$-martingale.
\end{proposition}

Conversely, the notion of martingale problems and the theorem of uniqueness make the connection from martingales to Markov processes. The following definitions and theorems come from the book of Ethier Kurtz\cite{EK} (chapter 4 section 3 and 4)
\begin{definition}
	Let $\mathcal{A}$ be a generator. By a solution of the martingale problem for $(\mathcal{A},\nu)$ we mean a càd-làg stochastic process $W$ with value in $E$ such that $\mathcal{L}(W(0)) = \nu$ and for each $f \in \overline{C}(E)$: 
	\[
	f(W(t)) - \int\limits_{0}^{t} \mathcal{A}f(W(s)) \text{d}s
	\]
	is a martingale with respect to the filtration $(\mathcal{F}_t^W)_t$ generated by $W$.
	%\[
	%	*\mathcal{F}_t^W:= \mathcal{F}_t^W %\vee \sigma\left( \int\limits_0^s h(W(u)) \text{d}u: s\leq t ,h \text{ bounded on $E$} \right)
	%\] (with the $(\mathcal{F}_t^W)_t$ be a filtration adapted to $W$).
\end{definition}

One can show (see the beginning chapter 3 section 3 of Ethier Kurtz\cite{EK}) that the statement "a càd-làg process $W$ is a solution of the martingale problem for $(\mathcal{A},\mu)$" is actually a statement about the distribution of $W$.

\begin{definition}
	We say that uniqueness holds for solutions of the martingale problem for $(\mathcal{A},\nu)$ if any two solutions of the martingale problem have the same finite-dimensional distributions.
\end{definition}
The following theorem gives a criterion to have uniqueness for martingale problem. Also he says that if there is uniqueness then the process is Markovian.
\begin{theorem}\label{Uniqueness Martingale problem}
	Let $E$ be a separable metric space and $\mathcal{A}$ be a linear operator of $\overline{C}(E)$ ($\overline{C}(E)$ is endowed with the uniform norm) such that 
	\begin{itemize}
		\item for each $f \in \overline{C}(E)$ and $\lambda >0$ we have: $\Vert \lambda f - \mathcal{A}f \Vert \geq \lambda \Vert f \Vert$. 
		\item the range of $(\lambda \text{Id} - \mathcal{A})$ verifies $\overline{\mathcal{R}(\lambda \text{Id} - \mathcal{A})} =  \overline{C}(E)$ for the uniform norm (for some $\lambda >0$).
	\end{itemize}
	Let $\nu$ be a probability measure on $E$ and suppose $W$ is a solution of the martingale problem for $(\mathcal{A},\nu)$.\\
	Then $W$ is a Markov process with generator $\mathcal{A}$, with initial distribution $\nu$, and uniqueness holds for the martingale problem for $\mathcal{A}$.
\end{theorem}

\vspace{10 pt}

Let us remind the generators used in the following.

\textbf{Notations}
\begin{itemize}
	\item 	We call $E_1:= \left( (\Z / m \Z)^2 \right)^N$ the finite state space of the Markov process $(X(t))_t:= (X^1(t),\dots,X^N(t))_t$.
    \item We call $E_2:= (\R_+ \times \lbrace -1,0,1 \rbrace)^N$ the state space of the stochastic process $ (Y(t),Z)_t:= ((Y^1(t),Z^1),\dots,(Y^N(t),Z^N))_t$.
    \item We denote $E:= \left( (\Z / m \Z)^2 \right) \times \R_+ \times \left(\lbrace -1,0,1 \rbrace \right)^N$ and we identify it with $\left( (\Z / m \Z)^2 \right)^N \times \R_+^N \times \lbrace -1, 0 ,1 \rbrace ^N$ and then we denote $(x,y,z) \in E$ with $x \in \left( (\Z / m \Z)^2 \right)^N$, $y \in \R_+^N$ and $z \in \left(\lbrace -1,0,1 \rbrace\right)^N$
    \item We denote by $\mathcal{C}_N = \lbrace (i,j) \in \lbrace 1,\dots,N\rbrace / i \neq j \rbrace$ all the particles couples which can play together
\end{itemize}
\textcolor{white}{text}\\

The generator of the Markov process we study is: for all $f \in \overline{C}(E)$,
\begin{eqnarray}\label{A}
\mathcal{A} f(x,y,z) = \mathcal{A}_d f(x,y,z) + v\,\mathcal{A}_g f(x,y,z).
\end{eqnarray}

The part $\mathcal{A}_d$ is the generator representing the motion of the particles. The particles move following independent random walks. Then $\mathcal{A}_d$ is defined by: for all 
 $f\in \overline{C}(E)$: 
\begin{eqnarray}\label{Ad}
\mathcal{A}_d f(x,y,z) = \sum\limits_{i=1}^{N} \sum\limits_{c=1}^2 \sum\limits_{\epsilon = \pm 1} \frac{d}{2 \times 2} [f(x + \epsilon e_i^c,y,z) - f(x,y,z)].
\end{eqnarray}
The generator $\mathcal{A}_g$ is the generator representing the evolution of the wealth of particles through games. Let us give firstly $\mathcal{A}_g$ then the explication of how it works. We have for all functions $f\in \overline{C}(E)$: $\mathcal{A}_g f(x,y,z)$ is equal to
\begin{eqnarray}\label{Ag}
  \sum\limits_{\substack{(i,j) \in \mathcal{C}_N}} \frac{1}{2} \mathds{1}_{x_i = x_j} \mathds{1}_{\substack{y_i > 0\\y_j >0}} \left[\begin{array}{l}\ \mathds{1}_{\substack{z_i = 0\\z_j =0}} (f(x,y + \mathbf{R} e_i + \mathbf{R} e_j,z) - f(x,y,z)) + \\
2\mathds{1}_{\substack{z_i = 1\\z_j = 0}} (f(x,y + \mathbf{T} e_i - \mathbf{S} e_j,z) - f(x,y,z)) +  \qquad\smallskip \\
\ \mathds{1}_{\substack{z_i = 1\\ z_j =1}} \left(\begin{array}{l}\ \frac{1}{2}(f(x,y -2\mathbf{P} e_j,z) - f(x,y,z))  \\
+ \frac{1}{2}(f(x,y -2\mathbf{P} e_i,z) - f(x,y,z))\end{array}\right) \ 

\end{array} \right]
\end{eqnarray}
Let us explain a bit this generator:
\begin{itemize}
	\item $\mathds{1}_{x_i = x_j}$ represents the spatial structure meaning that a game makes the wealth change only if the two particles are on the same site.
	\item $\mathds{1}_{\substack{y_i > 0\\ y_j >0}}$ checks if both particles are alive (\emph{i.e.} their wealth are positive)
	\item the big bracket represents the change of wealth through the game of the two particles. The term $\mathds{1}_{\substack{z_i=1\\z_j=0}}$ looks at the strategies of the players and choose the right payoff to give to the particles. For example if particle 1 and 2 are playing (both are alive and on the same site) if particle 1 cooperates and 2 defects then particle 1 gets a payoff of $-\mathbf{S}$ and particle 2 gets $\mathbf{T}$, the term in the generator representing this type of transition is: 
	\[
	\mathds{1}_{\substack{z_1=0\\z_2=1}} [f(x,y -\mathbf{S}e_1 + \mathbf{T}e_2,z) - f(x,y,z)].
	\]
\end{itemize}

\section{Convergence to the Random matching model}

In this section we show the following theorem:
\begin{theorem}\label{Thm Spatial model to Intermediate model}
\begin{comment}
	\begin{itemize}
	\item Let $d\in \N^*$.
	\item Let $m,N \in \N^*$, and let $v>0$.
	\item  The payoff matrices are (with $\mathbf{T}> \mathbf{R}>0$ and $\mathbf{S} > \mathbf{P} > 0 $): with probability 1/2
	\begin{eqnarray}
	\left(\begin{array}{cc}
	(\mathbf{R},\mathbf{R}) & (-\mathbf{S},\mathbf{T}) \\ (\mathbf{T},-\mathbf{S}) & (-2\mathbf{P},0) 
	\end{array}\right).
	\end{eqnarray}
	and with probability 1/2 the payoff matrix is:
	\begin{eqnarray}
	\left(\begin{array}{cc}
	(\mathbf{R},\mathbf{R}) & (-\mathbf{S},\mathbf{T}) \\ (\mathbf{T},-\mathbf{S}) & (0, -2\mathbf{P})
	\end{array}\right).
	\end{eqnarray}
	
	\end{itemize}
\end{comment}

Let $(X^d(t),Y^d(t),z)_t$ be a stochastic process with initial distribution $\nu \in \mathcal{P}(E)$ with second marginal denoted $\nu_y$ and generator $\mathcal{A}$ defined in $(\ref{A}),\ (\ref{Ad})$, $(\ref{Ag})$.\\

	Then $(Y^d,z)_d$ converges in distribution (as a sequence of stochastic processes) to the Markov process $(\overline{Y},z)$ with infinitesimal generator $\overline{\mathcal{A}}$ and domain $\overline{C}(E_2)$: for all $f\in \overline{C}(E_2)$
	\[
\overline{\mathcal{A}}f(y,z) = \frac{v}{m^2} \overline{\mathcal{A}_g} f (y,z)
	\]
	where $\overline{\mathcal{A}_g} f (y,z)$ is equal to
	\begin{eqnarray}\label{Generateur Modele intermediaire}
	\frac{1}{2}\sum\limits_{\substack{(i,j) \in \lbrace 1,\dots N \rbrace^2\\
			i \neq j}}  \mathds{1}_{\substack{y_i > 0\\y_j >0}} \left[\begin{array}{l}\ \mathds{1}_{\substack{z_i = 0\\z_j =0}} (f(y + \mathbf{R} e_i + \mathbf{R} e_j,z) - f(y,z)) + \\
	2\mathds{1}_{\substack{z_i = 1\\z_j = 0}} (f(y + \mathbf{T} e_i - \mathbf{S} e_j,z) - f(y,z)) +  \qquad\smallskip \\
	\ \mathds{1}_{\substack{z_i = 1\\ z_j =1}} \left(\begin{array}{l}\ \frac{1}{2}(f(y -2\mathbf{P} e_i,z) - f(y,z))  \\
	+ \frac{1}{2}(f(y -2\mathbf{P} e_j,z) - f(y,z))\end{array}\right) \ 
	
	\end{array} \right].
	\end{eqnarray}

\end{theorem}

\subsection*{Proof of the theorem}

To prove the previous theorem we use a homogenization theorem from Kurtz \cite{kurtz1992averaging}. 
\subsection{Preliminary lemma}

Before applying Kurtz's theorem, we will prove a lemma.\\

Observe that $X^d = (X^d(t))_{t\geq 0}$ (defined in the model) is a Markov process and its unique invariant distribution $\pi = (\mathcal{U}(\Z / m \Z)^2)^{\otimes N}$ is independent from $d$ (where $\mathcal{U}(\Z / m \Z)^2$ is the uniform law on the torus $(\Z / m \Z)^2$).
\begin{lemma}\label{Temps d'occup1}
	
	Let $t >0$ and $S$ be a Borel set of $E_1$.
Let $\Gamma^d([0,t]\times S) = \int_0^t \mathds{1}_{X^d(s) \in S} \text{d}s$. Then we have the following convergence: 
\[
\Gamma^d([0,t] \times S) \underset{d \to + \infty}{\overset{\mathcal{L}}{\longrightarrow}}t\, \pi(S).
\]
\end{lemma}
\begin{proof}

	Since $(X^1(td))_t$ and $(X^d(t))_t$ have the same generator and the same initial condition then $(X^1(td))_t \overset{\mathcal{L}}{=} (X^d(t))_t$.
	
	But by the ergodic theorem for Continuous Time Markov Chain, ($(X^1(s))_s$ is ergodic as a random walk on a finite torus) we have that: 
	\[
	t\frac{1}{d\,t}\int_{0}^{d\,t} \mathds{1}_{X^1(s)\in S} \text{d}s \underset{d \to +\infty}{\overset{a.s}{\longrightarrow}} t\,\pi(S).
	\]
	And finally we get: 
	\[
	\Gamma^d([0,t] \times S) \underset{d \to + \infty}{\overset{\mathcal{L}}{\longrightarrow}}t\, \pi(S).
	\]
	
\end{proof}

\subsection{Kurtz theorem}
Let us introduce some notations. We denote by $\ell_b(E_1)$ the set of locally finite measures $\mu$ on $[0, +\infty) \times E_1$ such that $\mu([0,t] \times E_1) = t$. We endow $\ell_b(E_1)$ with the metric $\hat{\rho}$ defined as follows. For $\mu \in \ell_b(E_1)$, we denote $\mu^t = \mu_{|[0,t]\times E_1}$ and let $\rho^t$ be the Prokhorov distance on the space of finite measures of $[0,t]\times E_1$. Then we set
\[
\hat{\rho}(\mu,\nu) := \int_{0}^{+\infty} \inf(e^{-t},\rho^t(\mu^t,\nu^t))\text{d}t.
\]

\begin{theorem}\label{Homogen Kurtz}
	Let $E_1$ and $E_2$ be two complete, separable metric spaces, and let $E = E_1 \times E_2$. For each $d\in \N^*$, let $(X^d(t),Y^d(t))_t$ be a stochastic process with sample paths in $D(\R_+,E)$ adapted to a filtration $(\mathcal{F}_t^d)_t$.\\
	\begin{itemize}
		\item [(A1)] Assume that for each $\epsilon  >0$, $T>0$, there exists a compact set $K \subset E_2$ such that: 
		\[
		\inf_{d > 0} \PP\lbrace \forall t\leq T,\  Y^d(t) \in K \rbrace \geq 1 - \epsilon.
		\]
		\item[(A2)] Assume that $\lbrace \mathcal{L}(X^d(t)): t \geq 0,\  d\in \N^* \rbrace$ is relatively compact (for the Prokhorov metric).
		\item[(A3)] Suppose there is an operator $\mathcal{A}: \mathcal{D}(\mathcal{A}) \subset \overline{C}(E_2) \to C(E_1 \times E_2)$ such that for $f \in \mathcal{D}(\mathcal{A})$ there is a process $\epsilon_d^f$ for which: \[
		f(Y^d(t)) - \int\limits_0^t \mathcal{A}f(X^d(s),Y^d(s))\text{d}s + \epsilon^f_d(t)
		\] is a $(\mathcal{F}^d_t)_t$-martingale.
		
		\item[(A4)] Suppose that for each $f \in \mathcal{D}(\mathcal{A})$ and for each $T >0$ there exists $p>1$ such that: \[
		\sup_{d\in \N^*} \E\left[ \int\limits_0^T \left| \mathcal{A}f(X^d(t),Y^d(t))\right|^p \text{d}t\right] < +\infty.
		\]
		\item[(A5)] Suppose that for each $f \in \mathcal{D}(\mathcal{A})$ and for each $T>0$ \[
		\lim_{d \to + \infty} \E[\sup_{t \leq T} |\epsilon^f_d(t)|] = 0.
		\]
	\end{itemize}
	Let $\Gamma^d$ be such that: (for every  Borel set $S$ of $E_1$) $\Gamma^d ([0,t] \times S) = \int\limits_0^t \mathds{1}_{X^d(s) \in S} \text{d}s $.\\
	Then we get that $(Y^d,\Gamma^d)_d$ is relatively compact in $D(\R_+,E_2) \times \ell_b(E_1)$.\\
	For any $(Y,\Gamma)$ limit point of $(Y^d,\Gamma^d)_d$, there exists a filtration $(\mathcal{G}_t)_t$ defined by:
	\[
	\mathcal{G}_t = \sigma\left(Y(s),\Gamma([0,s]\times H): s \leq t, H \text{ Borel set of }E_2\right)
	\] such that: for all $f \in \mathcal{D}(\mathcal{A})$ 
	\[
	f(Y(t)) - \int\limits_{[0,t] \times E_1} \mathcal{A}f(x,Y(s)) \Gamma(\text{d}s \times \text{d}x)
	\]
	is a $(\mathcal{G}_t)_t$-martingale.
\end{theorem}

\subsection{End of the proof}

First let us check the different assumptions of the previous theorem.\\
\begin{itemize}
	\item (A1): 
	Let $T>0$ and $\varepsilon>0$. We use Definition \ref{Construct Probabiliste Process Markov } to $\mathcal{A}_g$ we have that the rate matrix giving $\mathcal{A}_g$ is: $(\mathcal{A}_g((y,z),(y',z')))_{(y,z),(y',z') \in E_2}$ such that: for all $((y_1,z_1),\dots, (y_N,z_N)) = (y,z) \neq (y',z') = ((y'_1,z'_1),\dots, (y'_N,z'_N)) \in E_2$ \[\mathcal{A}_g((y,z),(y',z')) = \frac{1}{2} \mathds{1}_{x_i = x_j} \mathds{1}_{\substack{y_i > 0\\ y_j >0}} H((y,z),(y',z'))\] where 
	\[
	H((y,z),(y',z')) = \left\lbrace \begin{array}{ll}
\mathds{1}_{z_i = z_j = 0} & \text{if } y'=y + \mathbf{R}e_i + \mathbf{R}e_j\\
2\mathds{1}_{z_i = 1, z_j = 0} & \text{if } y' = y + \textbf{T}e_i - \textbf{S}e_j \smallskip\\
\frac{1}{2} \mathds{1}_{z_i = z_j = 1} & \text{if } y' = y -2\textbf{P} e_i \medskip\\
\frac{1}{2} \mathds{1}_{z_i = z_j = 1} & \text{if } y' = y - 2\textbf{P} e_j
	\end{array}\right.
	\]
	
	We notice that for all $(y,z) \in E_2$ $|\mathcal{A}_g((y,z),(y,z))| \leq \frac{vN^2}{2}$. This means that the jumping rates of $Y^d$ are bounded by $\frac{vN^2}{2}$ which doesn't depend neither on $d$ nor $t$ (because only $\mathcal{A}_g$ makes $Y^d$ change). Then the number of jumps of $Y^d$ until time $t$ denoted $\mathcal{N}_t^d$ is stochastically upper bounded as a process by a Poisson process denoted $(\mathcal{N}_t)_t$ with parameter $ \frac{vN^2}{2}$.
	
	Since we have: for all $n \in \N^*$
	\[
	\PP(\mathcal{N}_T > n) \leq \sum\limits_{k\geq n} \frac{(vTN^2)^k}{2^k k!}e^{-\frac{vTN^2}{2}} \underset{n \to + \infty}{\longrightarrow} 0
	\]
	Then $\exists n_\varepsilon \in \N^*$ such that \begin{eqnarray}\label{Proof Homogen Taux fixed implique nombre de saut controle}
	\PP(\mathcal{N}_T \leq n_\varepsilon) > 1 - \frac{\varepsilon}{2}.
	\end{eqnarray} Moreover at each jump the variation of $Y^d$ is bounded. Let us denote $ M:= \max(2\mathbf{P},\mathbf{R},\mathbf{S},\mathbf{T})$ this bound which is independent from $t,T,d$.\\
	Moreover since $\nu_y$ is probability measure in $E_2 = (\R_+ \times \lbrace 0,1 \rbrace)^N$ then $\nu_y$ is tight and it exists a compact $[\alpha,\beta] \subset \R_+$ such that: \begin{eqnarray}\label{Proof Homogen proa implique contenu dans un compact}
	\nu_y([\alpha,\beta] \times \lbrace0,1 \rbrace) \geq 1 - \frac{\varepsilon}{2}.
	\end{eqnarray}
	 And we have that \[\PP(\forall t \in [0,T],\  (Y^d(t),z) \in [\alpha - n_\varepsilon M ,\  \beta + n_\varepsilon M] \times \lbrace 0,1\rbrace )\] is upper bounded by: 
	 \[
	 \PP(Y^d(0) \in [\alpha,\beta] \text{ and } \forall t \in [0,T],\  (Y^d(t),z) \in [Y^d(0) - n_\varepsilon M ,\  Y^d(0) + n_\varepsilon M] \times \lbrace 0,1\rbrace ).
	 \]
	 Then by conditioning by the initial wealth, this last expression is equal to: 
	 \[
	 \PP(Y^d(0) \in [\alpha,\beta], \ \mathcal{N}^d_T \leq n_\varepsilon) \ = \  \ \PP(Y^d(0) \in [\alpha,\beta])\, \PP(\mathcal{N}^d_T \leq n_\varepsilon | Y^d(0) \in [\alpha,\beta])
	 \]
	 But since for all $y_0 \in \R_+$ $\Vert\mathcal{A}_g(y_0,.)\Vert_\infty \leq \frac{N^2 v}{2}$ we have that 
	 \[	 
	 \PP(\mathcal{N}^d_T \leq n_\varepsilon | Y^d(0) \in [\alpha,\beta]) \leq \PP(\mathcal{N}_T \leq n_\varepsilon).
	 \]
	 Finally using (\ref{Proof Homogen Taux fixed implique nombre de saut controle}) and (\ref{Proof Homogen proa implique contenu dans un compact}) we have 
	 \[
	 \inf_{d \in \N} \PP(\forall t \in [0,T],\  (Y^d(t),z) \in [\alpha - n_\varepsilon M ,\  \beta + n_\varepsilon M] \times \lbrace 0,1\rbrace) \geq (1 - \frac{\varepsilon}{2}) (1 - \frac{\varepsilon}{2}) \geq  (1 - {\varepsilon})
	 \]

	So we have (A1).
	\item (A2): Since $\forall d\in \N^*,\, \forall t \geq 0, \  X^d(t)$ is a random variable taking value in $E_2$ compact. Then Prokhorov theorem gives that  $\lbrace \mathcal{L}(X^d(t)), d\in \N^*,t  \geq 0 \rbrace$ is a relatively compact set of probability measures and we get (A2).
	\item (A3): Since for each positive integer $d$, $(X^d(t),Y^d(t),z)_t$ is a Markovian process with generator $\mathcal{A}$ and domain $\overline{C}(E)$, then we have for $f \in \overline{C}(E)$:
	\[
	f(X^d(t),Y^d(t),z) - f(X^d(0),Y^d(0),z) - \int\limits_0^t \mathcal{A} f (X^d(s),Y^d(s),z) \text{d}s
	\]
	is a $(\mathcal{F}_t^d)_t$-martingale.\\
	In particular $f(X^d(t),Y^d(t),z) - \int\limits_0^t \mathcal{A} f (X^d(s),Y^d(s),z) \text{d}s$
	is a $(\mathcal{F}_t^d)_t$-martingale.\\
	And if $f$ is constant on the first variable we get: $f \in \overline{C}(E_2)$, $\mathcal{A} f \in \overline{C}(E_1 \times E_2)$ and 
	\[
	f(Y^d(t),z) - \int\limits_0^t \mathcal{A} f (X^d(s),Y^d(s),z) \text{d}s
	\] is a $(\mathcal{F}_t^d)_t$-martingale.
	So we have (A3) with $\epsilon^f_d =0$.
	\item (A4): $\forall f \in \overline{C}(E)$ constant on the first variable and for all $x \in E_1,(y,z) \in E_2$ we have: \[\left|\mathcal{A}f(x,y,z)\right|^p  \leq (N^2 \frac{v}{2} \,8||f||_\infty)^p.\]
	So we get $\forall f \in \overline{C}(E)$ constant on the first variable: 
	\[
	\sup_{d >0} \E(\int\limits_0^T \left|\mathcal{A} f(X^d(t),Y^d(t),z)\right|^p \text{d}s) \leq (N^2 \frac{v}{2} \,8||f||_\infty)^p \,T < +\infty.
	\]
	So we get (A4).
	\item (A5):We have $\epsilon^f_d \equiv 0$, so we have (A5).
\end{itemize}
\textbf{Conclusion}\\
Applying Kurtz's theorem, we get that $((Y^d,z),\Gamma^d)_d$ is relatively compact with $\Gamma^d$ such that $\Gamma^d([0,t] \times S) = \int\limits_0^t \mathds{1}_{X^d(s) \in S} ds$ and for all limit points $((\overline{Y},z),\overline{\Gamma})$ there exists a filtration $(\mathcal{G}_t)_t$ such that $\forall f \in \overline{C}(E_2)$: 
\[
f(\overline{Y}(t),z) - \int\limits_0^t \int\limits_{E_1} \mathcal{A} f (x,\overline{Y}(s),z) \overline{\Gamma} (ds \times ds)
\]
is a $(\mathcal{G}_t)_t$-martingale.\\

But $\left( X_t^d \right)_t$ is a continuous-time simple symmetric random walk with jump speed $d$, so applying Lemma \ref{Temps d'occup1}, we get the following convergence: for all $t>0$, $S$ Borel set of $E_1$:
\[
\Gamma^d([0,t] \times S) \overset{\mathcal{L}}{\underset{d \to +\infty}{\longrightarrow}}  t\,\pi(S).
\]
and finally
\[
f(\overline{Y}(t),z) - \int\limits_0^t \int\limits_{E_1} \mathcal{A} f (x,\overline{Y}(s),z) \pi(\text{d}x)\text{d}s
\]
is a $(\mathcal{G}_t)_t$-martingale.\\
To finish we check the hypothesis of Theorem \ref{Uniqueness Martingale problem} to get that $(\overline{Y}(t),z)_t$ is a Markov process with generator 
\[
\overline{\mathcal{A}}f(y,z) = \int\limits_{E_1} \mathcal{A}f(x,y,z)\,\pi(\text{d}x) =  \frac{v}{2\, m^2} \mathcal{A}_g f(y,z)
\] and with domain $\overline{C}(E_1)$ and we have that uniqueness holds for the martingale problem.\\
Since $\overline{\mathcal{A}}$ is a bounded operator (bounded by $\frac{4 v N^2}{2 m^2}$) we have that for a $\lambda >  \frac{4 v N^2}{2 m^2}$, $|||\overline{\mathcal{A}}/\lambda ||| < 1$ and we get: \[
(\lambda Id - \mathcal{A})^{-1} = \frac{1}{\lambda} \sum_0^{+ \infty} \frac{\mathcal{A}^k}{k!}.
\]
So we have $\mathcal{R}(\lambda Id - \overline{\mathcal{A}}) = \overline{C}(E_1)$. Moreover $\overline{A}$ satisfies the positive maximum principle so using Lemma 2.1 p164 of Ethier Kurtz \cite{EK}, we have that for each $f \in \mathcal{D}(\overline{\mathcal{A}})$ and $\lambda>0$ we have: $\Vert \lambda f - \overline{\mathcal{A}}f \Vert \geq \lambda \Vert f \Vert$. So we have the hypothesis of Theorem \ref{Uniqueness Martingale problem}.
Since we have uniqueness for the martingale problem, we have that $(\overline{Y},z)$ is the unique weak limit point of $(Y^d,z)_d$. Finally we have that $(Y^d,z)_d$ converges in distribution to $(\overline{Y},z)$.

\section{Convergence to a mean field model}
\subsection{Framework and goal}

In this section we take the model from Theorem \ref{Thm Spatial model to Intermediate model}. The goal of this section is to show that when the number of players goes to $+\infty$, the wealth and the strategy of one particle is a non linear Markov process not depending on the wealth and strategies of the other particles. This kind of result is a propagation of chaos result. Before stating the theorem let us begin with some definitions and propositions.

Let us begin with one of the main notions which is the definition of $u$-chaotic sequence (from Saint Flour course of Sznitman \cite{sznitman1991topics})

\begin{definition}
    Let $F$ be a separable metric space and $(u_N)_N$ a sequence of probability measures such that $\forall N \in \N^*,$ $u_N$ is a symmetric probability on $F^N$. We say that $(u_N)_N$ is $u$-chaotic, with $u$ a probability measure on $F$, if for every $\phi_1,\dots, \phi_k$ continuous bounded functions of $F$ ($k \geq 1$),
	\begin{eqnarray}
	\lim\limits_{N \to +\infty} \int\limits_{F^N} \phi_1(x_1) \dots \phi_k(x_k) u_N(\text{d}x_1,\dots,\text{d}x_N) = \prod\limits_{i=1}^k \int\limits_F \phi_i(x) u(\text{d}x)
	\end{eqnarray}
\end{definition}
The meaning of this definition is if we look at a finite number of particles fixed, when the total number of particles goes to $+\infty$ these particles become independent.\\
To prove that a sequence is $u$-chaotic, Sznitman \cite{sznitman1991topics} also gives this useful proposition.

\begin{proposition}[ Prop 2.2 \cite{sznitman1991topics}]\label{2.2 Sznitman}
 Let $F$ be a metric separable space, and $(u_N)_N$ a sequence of symmetric probability measures on $F^N$.\\
	For all $N \in \N^*$, let $(X^N_1,\dots X^N_N)$ be a random vector with distribution $u_N$. 
	Let $u$ be a probability measure on $F$.\\
	
	We denote $\mu^N = \frac{1}{N} \sum_{i =1 }^{N} \delta_{X^N_i}$.
	\begin{enumerate}
		\item $(u^N)_N$ is $u$-chaotic if and only if
		\[\mathcal{L}(\mu^N) \underset{N \to +\infty}{\longrightarrow} \delta_u \qquad \text{ weakly in } \mathcal{P}(\mathcal{P}(F)) .\]
		\item When $F$ is Polish, $(\mathcal{L}(\mathcal{\mu}^N))_N$ is tight in $\mathcal{P}(\mathcal{P}(F))$ if and only $(\mathcal{L}(X_1^N))_N$ is tight.
	\end{enumerate}
\end{proposition}

To use the second item of the previous proposition we will need a tightness criterion for Markov processes. The two following theorems from Ethier Kurtz \cite{EK}, give a relative compactness criterion (for processes and in particular for Markov processes) using a martingale approach.
\begin{theorem}[Thm 3.9.1 \cite{EK}]\label{3.9.1 EK}
	Let $(F,r)$ be a Polish space, and let $(X_\alpha)_\alpha$ be a family of processes with sample paths in $D(\R_+,F)$. Suppose the compact containment condition holds \emph{i.e.} for every $\eta>0$ and $T >0$ there exists a compact set $K$ of $F$ such that:
	\[
	\inf_\alpha \PP(\forall t \leq T , X_\alpha(t) \in K) \geq 1 - \eta
	\]
	Let $H$ be a dense subset of the set of continuous bounded functions from $F$ to $\R$ for the topology of uniform convergence on compact sets. \\
	Then \\
	$(\mathcal{L}(X_\alpha))_\alpha$ is relatively compact if and only if $(\mathcal{L}(f \circ X_\alpha))_\alpha$ is relatively compact for each $f\in H$.
\end{theorem}

\begin{theorem}[Thm 3.9.4 \cite{EK}]\label{3.9.4 EK}
	Let $(F,r)$ be an arbitrary Polish space.\\
	For each $\alpha$, let $X_\alpha$ be a process with sample path in $D(\R_+,F)$ defined on a probability space $(\Omega_\alpha,\mathcal{F}^\alpha,P_\alpha)$ and adapted to a filtration $(\mathcal{F}^\alpha_t)_t$. Let $\mathcal{P}_\alpha$ be the Banach space of bounded real-valued $(\mathcal{F}_t^\alpha)_t$-progressive processes with norm $\Vert Y \Vert = \sup_{t \geq 0} \E[|Y(t)|]<+\infty$.
	Let 
	\[
	\hat{\mathcal{A}}_\alpha = \left\lbrace (Y,Z) \in \mathcal{P}_\alpha \times \mathcal{P}_\alpha: \quad Y(t) - \int_0^t Z(s) \text{d}s  \text{ is an $(\mathcal{F}^\alpha_t)_t$-martingale}  \right\rbrace
	\]
	Let $C_a$ be a subalgebra of $\hat{C}(F)$ (e.g. the space of bounded, uniformly continuous functions with bounded support), and let $D$ be the collection of $f \in \hat{C}(F)$ such that for every $\varepsilon > 0$ and $T>0$ there exists $(Y_\alpha,Z_\alpha) \in \hat{\mathcal{A}}_\alpha$ with 
	\[
	\sup_\alpha \E\left(\sup\limits_{t \in [0,T] \cap \mathbb{Q}}  \left|Y_\alpha(t) -f(X_\alpha(t)) \right|\right) < \varepsilon
	\]
	and
	\[
	\sup_\alpha \E\left(\int_0^T \left| Z_\alpha (t)  \right|^p \text{d}t \right) < +\infty \qquad \text{for some }p \in (1,+\infty)
	\]
	If $C_a \subset \bar{D}$ for the sup norm then $(\mathcal{L}(f \circ X_\alpha))_\alpha$ is relatively compact for each $f \in C_a$.
\end{theorem}

One of the main tool for the proof is the non linear martingale problem. The following proposition from Graham \cite{graham2000chaoticity}, gives us a contraction condition to have unicity for a non-linear martingale problem. 
\begin{proposition}[Prop 2.3 \cite{graham2000chaoticity}] \label{2.3 Graham}
	Let $F$ be a Polish space.\\
	For $q,q'$ signed measures we define the total variation norm  $\Vert q \Vert$ by:
	\[
	\Vert q \Vert = \sup\limits_{\phi \text{ s.t. } \Vert \phi \Vert_\infty \leq 1 } \left|\int_F \phi(x) q(\text{d}x)\right| 
	\]
	and the total variation distance $\Vert q - q'\Vert$ by:
	\[
	\Vert q - q'\Vert = \sup\limits_{\phi \text{ s.t. } \Vert \phi \Vert_\infty \leq 1 } \left|\int_F \phi(x) q(\text{d}x) - \int_F \phi(x) q'(\text{d}x)\right| 
	\]
	For $q\in \mathcal{P}(f)$, let $\mathcal{A}(q): \phi \in \mathbb{L}^\infty(F) \mapsto \mathcal{A}(q)\phi \in \mathbb{L}^\infty(F)$ be given for some uniformly bounded positive measure kernel $J$ on $\mathcal{P}(F) \times F$ by:
	\[
	\forall \phi \in \mathbb{L}^\infty, \quad (\mathcal{A}(q)\phi)(x) = \int_F (\phi(y) -\phi(x)) J(q,x,\text{d}y).
	\] 
	Let $Q \in \mathcal{P}(D(\R_+,F))$ and $Q_s \in \mathcal{P}(F)$ such that for all $B\subset F,$ $Q_s(F) = Q(X(s) \in F)$ and 
	\[
	\forall \phi \in \mathbb{L}^\infty(F), \quad M_t^\phi = \phi(X(t)) - \phi(X(0)) - \int_{0}^{t} (\mathcal{A}(Q_s)\phi)(X(s)) \text{d}s
	\] define a $Q$-martingale (where $X$ is the canonical process defined by $X(t)(x) = x(t)$), and it solves the non-linear martingale problem starting at $q$ if moreover $Q_0 = q$.\\
	
	Assume that for some constants $\beta >0$ and $\kappa >0$,
	\[
	\Vert J(q,x)\Vert  \leq \beta \quad \text{and} \quad \Vert J(p,x) - J(q,x)\Vert  \leq \kappa\Vert p -q\Vert , \quad \forall p,q \in \mathcal{P}(F),\ \forall x \in F
	\]
	
	Then for any $q \in \mathcal{P}(F)$, there is a unique solution $Q$ in $\mathcal{P}(D(\R_+,F))$ for the above non-linear martingale problem starting at $q$. The solutions depend continuously on $q$.
\end{proposition}

Now that we have the first tools we need, we can explain a bit more the model we consider then state the theorem we will prove.

\subsection{Statement of the theorem}
Before stating the theorem, let us reintroduce the model.\\
We use the model presented in Section \ref{Model} when $d$ goes to $+\infty$, so we consider the process introduced by Theorem \ref{Thm Spatial model to Intermediate model}. We can see this model either from the generator and initial distribution point of view, or with Poisson processes.\\
Theorem \ref{Thm Spatial model to Intermediate model} replaces when $d$ goes to $+\infty$: checking if  two particles have the same position before playing the game by running a Bernoulli variable of parameter $1/m^2$. If this variable is equal to 1 then $i$ and $j$ play a Prisoner's Dilemma Game, else nothing happens.\\

Let us denote $(Y^N(t),z^N)_t$ the Markov process where $N$ particles evolve according to the previous rules.\\
Let us recall that for a probability distribution on the sample path $\mu \in D(\R_+,\R_+ \times \lbrace 0,1 \rbrace)$ we denote for all Borel sets $S$ of $\R_+ \times \lbrace 0,1 \rbrace$ and for all $t\geq0$: \[
\mu_t(S)= \mu(X(t) \in S)
\]
In this part to have propagation of chaos, we will consider a model where we will slow the time. \\
Let $(\hat{Y}^N(t))_t$ be defined by $\hat{Y}^N(t)= Y^N(t/N)$ for all $t\geq0$ (we do not need to modify $z^N$ because $z^N$ is a constant process). So the generator of $(\hat{Y}^N,z^N)$ (with domain $\overline{C}(E_2)$) is: for all $(y,z) \in E_2$
	\begin{eqnarray} \label{Gene random matching Temps ralenti}
\overline{\overline{\mathcal{A}}}f(y,z) = \frac{1}{N}\frac{v}{m^2} \overline{\mathcal{A}_g} f (y,z)
	\end{eqnarray}	
with $\overline{\mathcal{A}_g}$ defined in (\ref{Generateur Modele intermediaire}).

\vspace{2pt}
\begin{theorem}\label{Thm proposition chaos Modèle inter to MF}
\begin{comment}
	\begin{itemize}
	\item Let $N \in \N^*$, and let $v>0$.
	\item 
	\item Let $w_0 \in \R_+^*$.
	\item  The payoff matrices are (with $\mathbf{T}> \mathbf{R}>0$ and $\mathbf{S} > \mathbf{P} > 0 $): with probability 1/2
	\begin{eqnarray}
	\left(\begin{array}{cc}
	(\mathbf{R},\mathbf{R}) & (-\mathbf{S},\mathbf{T}) \\ (\mathbf{T},-\mathbf{S}) & (-2\mathbf{P},0) 
	\end{array}\right).
	\end{eqnarray}
	and with probability 1/2 
	\begin{eqnarray}
	\left(\begin{array}{cc}
	(\mathbf{R},\mathbf{R}) & (-\mathbf{S},\mathbf{T}) \\ (\mathbf{T},-\mathbf{S}) & (0, -2\mathbf{P})
	\end{array}\right).
	\end{eqnarray}
	
	\end{itemize}

\end{comment}

Let $\nu$ be a probability distribution of $\R_+ \times \lbrace 0, 1 \rbrace$.\\
Let $(\hat{Y}^N(t),z^N)_{t\geq0}$ be a stochastic process (relative to a filtration $(\mathcal{F}_t)_t$) with generator $\overline{\overline{\mathcal{A}}}$ and initial distribution $\nu ^{\otimes N}$\\

Then there exists a probability measure $\mu$ on ${D}(\R_+,\R_+\times \lbrace0,1 \rbrace)$ such that:\\
 the sequence of probability distributions $(\mathcal{L}(\hat{Y}^N,z^N))_N$ is $\mu$-chaotic.\\
The distribution $\mu \in \mathcal{P}(D(\R_+,\R_+))$ has non-linear generators $(L_{\mu_t})_t$ with domain $\overline{C}(\R_+ \times \lbrace 0,1 \rbrace)$ and initial distribution $\nu$.\\
For all $f\in \overline{C}(\R_+\times \lbrace 0,1\rbrace)$, $y\in \R_+$ and $z \in \lbrace 0,1 \rbrace$, $L_{\mu_t} f(y,z)$ is equal to:

\begin{eqnarray}\label{Non linear generator MF}
\frac{v}{m^2} \frac{1}{2} \mathds{1}_{y>0}\left[\begin{array}{l}
\ \ \mathds{1}_{z=0}\, \mu_t((0,+\infty)\times \lbrace 0 \rbrace) (f(y + \mathbf{R},z) - f(y,z))\smallskip\\
+\mathds{1}_{z=0}\, \mu_t((0,+\infty)\times \lbrace 1 \rbrace) (f(y - \mathbf{S},z) - f(y,z)) \smallskip\\
+\mathds{1}_{z=1}\, \mu_t((0,+\infty)\times \lbrace 0 \rbrace) (f(y + \mathbf{T},z) - f(y,z)) \smallskip\\
+\mathds{1}_{z=1}\, \mu_t((0,+\infty)\times \lbrace 1 \rbrace) \frac{1}{2}(f(y - 2\mathbf{P},z) - f(y,z))
\end{array}\right].
	\end{eqnarray}

\end{theorem}
\vspace{4pt}
In order to use, and simulate the mean field process on an easier way, we also have the following evolution equations.

\begin{corollary}\label{Corol Evolution Equation}
	A process $(Y(t),z)_t$ whose trajectories have distribution $\mu$ (of generator $(L_{\mu_t})_t$ with domain $\overline{C}(\R_+ \times \lbrace0,1 \rbrace)$ and initial distribution $\nu$) satisfies the following evolution equations: with $\mathcal{L}(Y(0),z) = \nu$, for all $y \in \R_+^*$.
	
	\begin{eqnarray} \label{Evol B}
	\begin{array}{rcl}
	\frac{2 m^2}{v}\dro \PP(Y(t) = y,z =0) & = & \ \ \PP(Y(t) = y - \mathbf{R},z= 0)\, \PP(Y(t)>0,z=0)\mathds{1}_{y>\mathbf{R}}\smallskip\\
	& & + \PP(Y(t) = y + \mathbf{S}, z=0)\, \PP(Y(t) >0, z=1)\smallskip\\
	& & - \PP(Y(t)=y,z = 0)\,\PP(Y(t)>0)
	\end{array}
	\end{eqnarray}

	\begin{eqnarray} \label{Evol R}
	\begin{array}{rcl}
	\frac{2 m^2}{v} \dro \PP(Y(t) = y,z =1) & = &\ \ \ \  \PP(Y(t) = y - \mathbf{T},z= 1)\, \PP(Y(t)>0,z=0)\mathds{1}_{y >\mathbf{T}}\smallskip\\
	& & + \frac{1}{2}\PP(Y(t) = y + 2\mathbf{P}, z=1) \,\PP(Y(t) >0, z=1)\smallskip\\
	& & + \frac{1}{2}\PP(Y(t) = y, z=1)\, \PP(Y(t) >0, z=1)\smallskip\\
	& & - \ \ \PP(Y(t)=y,z = 1)\,\PP(Y(t)>0)
	\end{array}
	\end{eqnarray}
	
\end{corollary}

	\subsection{Preliminary work}
	We will adapt a proof of Bordenave and al. \cite{bordenave2007particle}.\\

    We will show that: there exists a probability measure $\mu$ on ${D}(\R_+,\R_+)$ such that the sequence of distributions $(\mathcal{L}(\hat{Y}^N))_N$ is $\mu$-chaotic.\\
    
   Let us denote $\mu^N \in \mathcal{P}(D(\R_+,\R_+))$ the empirical measure of $(\hat{Y}_i^N,z_i^N)_i$: 
    
    \begin{eqnarray}\label{Meas empirique proposition du chaos}
    \mu^N = \frac{1}{N} \sum_{i=1}^{N} \delta_{(\hat{Y}_i^N,z_i^N)}.
    \end{eqnarray}

   Let us begin with some lemmas:
    \begin{lemma}\label{Tight}
    The sequence of distributions $(\mathcal{L}(\hat{Y}_1^N,z^N))_N \subset D(\R_+,\R_+)^\N$ is tight.
    \end{lemma}
    \begin{proof}
    	To prove this lemma we will apply Theorems 3.9.1 and 3.9.4 of Ethier and Kurtz (see Theorem \ref{3.9.1 EK} and Theorem \ref{3.9.4 EK}), it remains to check the assumptions.
    	Since $\hat{Y}^N$ is a Markov process, when we take in (\ref{Gene random matching Temps ralenti}) functions constant in $(y_i,z_i^N)\ \  i\geq 2$ we get:
    	 
    	\begin{eqnarray}\label{Martingale pour Y_1^N}
    	   	f(\hat{Y}_1^N(t),z_1^N) - f(\hat{Y}_1^N(0),z_1^N) - \int_{0}^{t} L_{\mu^N_u}f(\hat{Y}_1^N(u-),z_1^N) \text{d}u + h(\hat{Y}^N_1,z_1^N).
    	\end{eqnarray}	
    is a $(\mathcal{F}_t)$-martingale (where $h$ is defined in \ref{Gene une particle dans random matching Intro}).\\
    	Since the jumping rates and amplitudes of $(\hat{Y}_1^N(t))_t$ are uniformly bounded (on $N$) we have the compact containment condition \emph{i.e.} $\forall \varepsilon >0$, $\forall\, T >0$ there exists a bounded subset $K \subset \N$ such that:
    	\[
    	\inf_N \PP(\forall t\leq T, (\hat{Y}_1^N(t),z^N) \in K) \geq 1 - \varepsilon
    	\]
    	Moreover, for all $p\in (0,+\infty)$ and all $T>0$ we have: $\sup_N \E\left( \int_0^T\left(L_{\mu^N_u}f(Y_1^N(u-),z_1^N)\right)^p\right)$ is equal to
    	\[
    	\sup_N \E\left( \int_0^T\left(L_{\mu^N_u}f(Y_1^N(u-),z_1^N)\right)^p\right) \leq \frac{v}{2 m^2}T\,2^p \Vert f \Vert_\infty^p < +\infty
    	\]
        Thus we apply Theorem 3.9.4, then Theorem 3.9.1 of Ethier Kurtz to get the tightness of $(\mathcal{L}(\hat{Y}_1^N),z^N)_N$.
    \end{proof}
    In order to make the proof easier to read let us give the following notations.
    		        		\textbf{Notations}
    		        		For $y \in \R_+$ and $z^N \in \lbrace 0,1 \rbrace^N$ we define for $(i,j) \in \lbrace 1,\dots, N \rbrace$:
    		        		\[
    		        		(y+g_{i,j},z^N) = \left\lbrace \begin{array}{ll}
    		        		(y + \mathbf{R},z^N)  & \text{if } z^N_i = 0, z^N_j = 0\\
    		        		(y + \mathbf{T},z^N)  & \text{if } z^N_i = 1, z^N_j = 0 \\
    		        		(y - \mathbf{S},z^N)  & \text{if } z^N_i = 0, z^N_j = 1 \\ 
    		        			\end{array}\right.
    		        		\]
    		        		
    		        		For the particular case $z^N_i = 1$ and $z^N_j = 1$ where the payoffs are random, the notation is: for $i,j \in \lbrace 1,\dots N \rbrace$ if $z^N_i = 1$ and $z^N_j = 1$
    		        		\[
    		        		[f(y + g_{i,j},z^N) -f(y,z^N)] = [\frac{1}{2}(f(y - 2 \mathbf{P},z^N) -f(y,z^N))]
    		        		\]
    		        		\newline
    The following lemma is one of the main keys to the propagation of chaos. It says that the correlations between two particles go to 0 when the number of particles goes to $+\infty$.
    \begin{lemma}\label{Controle Martingale}
    Let us denote $\mathcal{N}^{i,j}$ the  Poisson process attached to the couple of particles $(i,j)$. This Poisson process has intensity $\frac{v}{2m^2}\frac{1}{N}$ by (\ref{Gene random matching Temps ralenti}).\\
    Let us denote the compensated Poisson process $(\mathcal{M}^{i,j}(t))_t =(\mathcal{N}^{i,j}(t) - \frac{v}{2m^2}\frac{1}{N}t)_t$.\\
     For all $f \in \overline{C}(E_2)$, we denote by $M^{f,N}_i(t)$ the following:
     \begin{eqnarray}
   \sum_{\substack{j=1\\
     		j \neq i}}^{N}\int_{0}^{t} \mathds{1}_{\hat{Y}_i^N(u-)>0}\mathds{1}_{\hat{Y}_j^N(u-)>0}[f(\hat{Y}_i^N(u-)+g_{i,j}) - f(\hat{Y}_i^N(u-))] \text{d}\mathcal{M}^{i,j}(u).
     \end{eqnarray}
     Then $M_i^{f,N}$ is a square integrable martingale and for all $N>0$ and all $0< i \neq j \leq N$ there exists $C^f >0$ such that:
     \[
     \E(M_i^{f,N}(t)M_j^{f,N}(t)) \leq\frac{C^f}{N}.
     \] 
    \end{lemma}
    The intuition behind the following proof is to say that when the number of particles goes to $+\infty$ two fixed particles are less and less likely to interact with each other.
    \begin{proof}
    	In order for the proof to be easier to read, we will denote: \[
    	\forall u \geq 0 \qquad \Delta_{ij}(u) = [f(Y_i^N(u-)+g_{i,j},z^N) - f(Y_i^N(u-),z^N)]
    	\]
    Since $\mathcal{M}^{i,j}$ is a martingale, the stochastic integral $M^{f,N}_i$ is a martingale.
Moreover, since $f$ is bounded, $M_i^{f,N}$ is square integrable.\\
By exchangeability, we have to show that 
    for all $N\in \N^*$, there exists $C^f >0$ such that:
    \[
    \E(M_1^{f,N}(t)M_2^{f,N}(t)) \leq\frac{C^f}{N}.
    \] 
    By the product rule (Section 6.4 of Kurtz \cite{kurtz2001lectures}) we have (with $<\cdot,\cdot>$ is the quadratic covariation): $\E(M_1^{f,N}(t)\, M_2^{f,N}(t))$ is equal to
    \[
    \underbrace{\E(M_1^{f,N}(0)\, M_2^{f,N}(0))}_{=0} + \underbrace{\E(\int_0^t M_1^{f,N}(u-) \text{d}M_2^{f,N}}_{=0} + 0 + \E (<M_1^{f,N},M_2^{f,N}>_t).
    \]
    By bilinearity of the quadratic covariation and using Lemma 6.6 of Kurtz \cite{kurtz2001lectures}, we get that $\E(M_1^{f,N}(t)\, M_2^{f,N}(t)) $ is equal to 
    \[
    \sum_{j = 2}^{N} \sum_{\substack{i = 1\\
    		i \neq 2}}^{N} \E\left(\int_{0}^{t}  \Delta_{1i}(u) \Delta_{2j}(u) \mathds{1}_{Y_1(u)>0}\mathds{1}_{Y_j(u)>0} \mathds{1}_{Y_2(u)>0}\mathds{1}_{Y_i(u)>0}\  \text{d}<\mathcal{M}^{1,j}, \mathcal{M}^{2,i}>_u \right).
    \]
    Since $\mathcal{M}^{1,2} = \mathcal{M}^{2,1}$ and separating the term of the previous double sum in two big parts: the interaction between 1 and 2, and the other interactions we get: $\E(M_1^{f,N}(t)\, M_2^{f,N}(t))$ is equal to
    \[
    \begin{array}{cl}
   & \E\left(\int_{0}^{t} \Delta_{1j}(u) \Delta_{2j}(u) \mathds{1}_{Y_1(u)>0} \mathds{1}_{Y_2(u) > 0} \text{d}<\mathcal{M}^{1,2},\mathcal{M}^{1,2}>_u \right) \medskip\\
  & + \sum_{i,j = 3}^{N} \E \left(\int_{0}^{t} \Delta_{1j}(u) \Delta_{2i}(u) \mathds{1}_{Y_1(u)>0} \mathds{1}_{Y_2(u) > 0}\mathds{1}_{Y_i(u)>0} \mathds{1}_{Y_j(u) > 0} \text{d}<\mathcal{M}^{1,j}, \mathcal{M}^{i,2}>_u \right)\\
  & + \sum_{i=3}^{N} \E \left(\int_{0}^{t} \Delta_{1j}(u) \Delta_{2i}(u) \mathds{1}_{Y_1(u)>0} \mathds{1}_{Y_2(u)>0} \mathds{1}_{Y_i(u) > 0} \text{d}<\mathcal{M}^{1,2}, \mathcal{M}^{2,i}>_u \right)\\
  & + \sum_{j=3}^{N} \E \left(\int_{0}^{t} \Delta_{1j}(u) \Delta_{2i}(u) \mathds{1}_{Y_1(u)>0} \mathds{1}_{Y_2(u)>0} \mathds{1}_{Y_j(u) > 0} \text{d}<\mathcal{M}^{1,j}, \mathcal{M}^{
 	1,2}>_u \right).
    \end{array}  
    \]
    But for all $(i,j)\neq(i',j')$ $\mathcal{M}^{i,j}$ and $\mathcal{M}^{i',j'}$ are two independent martingales,
    so $\left(\mathcal{M}^{i,j}(t)\, \mathcal{M}^{i',j'}(t)\right)_t$ is a martingale for all $(i,j) \neq (i',j')$. Hence we get: $<\mathcal{M}^{i,j}, \mathcal{M}^{i',j'}>_u = 0$ for all $u >0$ and $(i,j) \neq (i',j')$.\\
    
    Finally we get (because $<\mathcal{M}^{1,2}>_t =\frac{v}{2m^2} \frac{t}{N}$): 
    \[
    \E(M_i^{f,N}(t) M_j^{f,N}(t)) \leq 4\,t\, \Vert f \Vert^2_\infty \frac{v}{2m^2}\frac{1}{N}.
    \]
    \end{proof}
    \subsection{Proof of Theorem \ref{Thm proposition chaos Modèle inter to MF}}
    Now we will prove the propagation of chaos.\\
    By Proposition 2.2 of Sznitman \cite{sznitman1991topics} (see Proposition  \ref{2.2 Sznitman}), we have to show that: there exists $\mu \in \mathcal{P}(D(\R_+,\R_+ \times \lbrace 0,1 \rbrace))$ such that: with $\mu^N$ defined in (\ref{Meas empirique proposition du chaos}) \\
    \[\mathcal{L}(\mu^N) \underset{N \to +\infty}{\longrightarrow} \delta_\mu \qquad \text{ weakly in } \mathcal{P}(\mathcal{P}(D(\R_+,\R_+ \times \lbrace 0,1 \rbrace))) .\]

    \textbf{Step 1: Relative compactness of the empirical measure}\\
    
    By Lemma \ref{Tight}, $(\mathcal{L}(Y^N_1))_N \in \mathcal{P}(D(\R_+,\R_+ \times \lbrace 0,1 \rbrace))$ is tight. By Proposition 2.2 of Sznitman \cite{sznitman1991topics} (see Proposition  \ref{2.2 Sznitman}) we have that:    $(\mathcal{L}(\mu^N))_N \subset \mathcal{P}(\mathcal{P}(D(\R_+,\R_+\times \lbrace 0,1 \rbrace)))$ is tight. So by Prohorov's theorem $(\mathcal{L}(\mu^N))_N$ has a converging subsequence.\\
    
    \textbf{Step 2: Convergence to a solution of a martingale problem}\\

    Let $\Pi^\infty$ be a limit point of $(\mathcal{L}(\mu^N))_N$ and let $\mu$ be a random variable in $\mathcal{P}(D(\R_+,\R_+))$ with distribution $\Pi^\infty$. 
   \\
    Before introducing the non-linear martingale problem, let us introduce some notations.

    \textbf{Notations}
    	Let $F$ be a Polish space, $\mu$ a probability measure and $\phi$  a bounded measurable function  from $F$ to $\R$. We set:
    	\[
    	\langle \mu , \phi \rangle = \int_F \phi(x) \mu(\text{d}x)
    	\]
    \newline

    We will prove the following lemma: 
     \begin{lemma}\label{Non linear Martingal problem MF}
     	$\mu$ satisfies a non linear martingale problem with initial distribution $\nu$. More precisely, for all $f$ continuous bounded, for all $T>t>0$ (with $(X,Z) = (X(t),Z(t))_t$ is the canonical process in $D(\R_+,\R_+\times \lbrace 0,1 \rbrace)$ and $L$ defined in $(\ref{Non linear generator MF})$)
     	\[
     	M^f(t) = f(X(t),Z(t)) -f(X(0),Z(0)) - \int_0^t L_{\mu_u}f (X(u),Z(u)) \text{d}u
     	\]
      	is a $\mu$-martingale and $\mu(0) = \nu \ \Pi^\infty$ a.s.
     \end{lemma}

     \begin{proof}\textcolor{white}{text}\\
     	We take $k \in \N^*$ and $0 \leq t_1 < t_2<\dots <t_k \leq t < T$, $g$ a continuous bounded function from $\R_+^k$ to $\R$ and $f$ a continuous bounded function from $\R_+$ to $\R$,  let $\mathcal{G}$ be  a function from $\mathcal{P}(D(\R_+,\R_+ \times \lbrace 0,1 \rbrace))$ to $\R$ defined by:
\[
\mathcal{G}(R) =  \left\langle R, \left( M^f(T) -M^f(t)\right)g((X(t_1),Z(t_1)),\dots,(X(t_k),Z(t_k))) \right\rangle.
\]
Using the same argument as in the proof of Theorem 4.5 of Graham and M\'el\'eard\cite{graham1997stochastic}, we have that for all $0 \leq t_1 < t_2<\dots <t_k \leq t < T$ outside a countable space denoted $D$, $\mathcal{G}$ is $\Pi^\infty$ a.s. continuous.\\

We will show that:\[\mathcal{G}(\mu) = 0 \quad \Pi^\infty \text{ a.s.}\] Indeed, if it is true for all $0 \leq t_1 < t_2<\dots <t_k \leq t < T$ outside of $D$ and for any $g$ continuous bounded from $\N^k$ to $\R$ then by the Monotone Class Theorem, $\forall A \subset \mathcal{T}_t: \ \langle \mu, M^f(T) \mathds{1}_A  \rangle = \langle \mu , M^f(t)\mathds{1}_A \rangle$. So $\mu$ satisfies the above non linear martingale problem.\\
It remains to show that $\mathcal{G}(\mu) = 0$ $\Pi^\infty$ a.s. \\
We have that $ \mathcal{G} \left(\frac{1}{N}\sum_{i=1}^{N} \delta_{(Y_i^N,z_i^N)}\right)$ is equal to
\[
 \frac{1}{N}\sum_{i=1}^{N}
\left(f(Y_i^N(T),z_i^N) - f(Y_i^N(t),z_i^N) - \int_{t}^{T} L_{\mu_u}f(Y_i^N,z_i^N) \text{d}u\right)g_i^N.
\]
where $g_i^N = g((Y_i^N(t_1),z_i^N),\dots,(Y_i^N(t_k),z_i^N))$.\\
We will take the notations of Lemma \ref{Controle Martingale}. We have: (since $f$ is bounded)
\[
\begin{array}{rcl}
f(Y_i^N(T)) - f(Y_i^N(t)) & = & \sum\limits_{\substack{j=1\\j\neq i}}^{N}\int\limits_{t}^{T} \mathds{1}_{Y_i^N(u)>0}\mathds{1}_{Y_j^N(u)>0} \Delta_{ij}(u)\text{d}\mathcal{N}^{i,j}(u) \medskip\\
& \leq & M_i^{f,N}(T) - M_i^{f,N}(t)\\
& &  + \int_{t}^{T} \left(L_{\mu^N_u} f(Y_i^N(u),z_i^N) +O\left(\frac{1}{N}\right) \right) \text{d}u.
\end{array}
\]

So $ \mathcal{G} \left(\frac{1}{N}\sum_{i=1}^{N} \delta_{Y_i^N}\right)$ becomes less or equal than:
\[
\begin{array}{l}
\frac{1}{N} \sum_{i=1}^{N} \left(M_i^{f,N}(T) - M_i^{f,N}(t)\right)g_i^N \\
 + \frac{1}{N}\sum_{i=1}^{N}\left(\int_{t}^{T} L_{\mu_u^N} f(Y_i^N,z_i^N) - L_{\mu_u} f(Y_i^N,z_i^N) +O\left(\frac{1}{N}\right)\text{d}u\right)g_i^N.
\end{array}
\]
Hence $\E(| \mathcal{G} \left(\frac{1}{N}\sum_{i=1}^{N} \delta_{Y_i^N}\right)|)$ is less than or equal to: 
\[
\begin{array}{l}
\overbrace{\E\left(\left|\frac{1}{N} \sum_{i=1}^{N} \left(M_i^{f,N}(T) - M_i^{f,N}(t)\right)g_i^N\right|\right)}^{A} \\
+ \underbrace{\E\left(\left|\frac{1}{N}\sum_{i=1}^{N}\left(\int_{t}^{T}L_{\mu_u^N} f(Y_i^N,z_i^N) - L_{\mu_u} f(Y_i^N,z_i^N) +O\left(\frac{1}{N}\right)\text{d}u\right)g_i^N\right|\right)}_{B}.
\end{array}
\]
$A$ is the term corresponding to the variation of the process. $B$ corresponds to the convergence of $\mu^N$ to $\mu$.
First, let us show that $A$ goes to 0 when $N$ goes to $+\infty$.
\[
\begin{array}{rcl}
A^2 & \leq & \frac{1}{N^2} \E\left(\left(\sum\limits_{i=1}^{N}(M_i^{f,N}(T) - M_i^{f,N}(t))g_i^N\right)^2\right)\medskip\\
 & \leq & \frac{1}{N^2}\sum\limits_{(i,j) \in \lbrace1,\dots,N \rbrace^2} \E\left( (M_i^{f,N}(T) - M_i^{f,N}(t))(M_j^{f,N}(T) - M_j^{f,N}(t))g_i^N g_j^N \right).
\end{array}
\]
By the exchangeability of $(Y^N_i)_i$ we get: 
\[
 \begin{array}{rcl}
 A^2 &\leq & 
\frac{1}{N} \Vert g_1 \Vert_\infty^2 \E \left((M_1^{f,N}(T) - M_1^{f,N}(t))^2 \right)\smallskip\\
& & + \frac{N(N-1)}{N^2} \E(g_1^N g_2^N (M_1^{f,N}(T) - M_1^{f,N}(t))(M_2^{f,N}(T) - M_2^{f,N}(t))).
\end{array}
\]
The second term of the sum is less than or equal to (because $g$ is bounded): 
\[
\Vert g_1 \Vert_\infty \, \Vert g_2 \Vert_\infty \left(\begin{array}{c}
 \E(M^{f,N}_1(T) M^{f,N}_2(T)) - \E(M^{f,N}_1(T)M^{f,N}_2(t))\\ - \E(M^{f,N}_1(t)M^{f,N}_2(T)) + \E(M^{f,N}_1(t)M^{f,N}_2(t))
\end{array} \right).
\]
By conditioning by $\mathcal{F}_t$ and applying Lemma \ref{Controle Martingale}  we get that: 
\[
A^2 \underset{N \to + \infty}{\longrightarrow} 0
\]

Now let us show that $B$ goes to 0 when $N \to +\infty$.\\
Since $\Pi^N \to \Pi^{\infty}$, we have $\mu^N$ converges in distribution when $N \to +\infty$ to $\mu$ so we have upper-bounding $g$ and $L$:

\[
B \leq \frac{v}{2m^2} \frac{7}{2}\Vert f \Vert_\infty \, \Vert g \Vert_\infty \E(\int_{t}^{T} (\mu_u^N((0,+\infty)) - \mu_u((0,+\infty)) + O\left(\frac{1}{N}\right) \underset{N \to +\infty}{\rightarrow} 0.
\]
To end the proof of the lemma, we use the Fatou's lemma to get:
\[
\E(|\mathcal{G}(\mu)|) \leq \lim\limits_{N \to +\infty} \E(|\mathcal{G}(\mu^N)|) = 0
\]
Hence we have $\mathcal{G}(\mu) = 0$.\\
Since $(X,Z) \in D(\R_+,\R_+ \times \lbrace 0,1 \rbrace) \mapsto (X(0),Z(0))$ is continuous, we have that $\mu(0) = \delta_\nu$, $\Pi^\infty$ a.s. which concludes the proof.
     \end{proof}

     \textbf{Step 3: Uniqueness of the solution of the martingale problem}
     
     We will use Proposition 2.3 of Graham \cite{graham2000chaoticity} (see Proposition \ref{2.3 Graham}) and the notations within. \\
     In our problem we have for $q$ a probability measure: \[
     J(q,(x,z))  =  \frac{v}{2m^2} \frac{1}{2} \mathds{1}_{x>0}\left[\begin{array}{l}
     \ \ \mathds{1}_{z=0}\, q((0,+\infty)\times \lbrace 0 \rbrace) \delta_{(x + \mathbf{R},z)}\smallskip\\
     +\mathds{1}_{z=0}\, q((0,+\infty)\times \lbrace 1 \rbrace) \delta_{(x - \mathbf{S},z)} \smallskip\\
     +\mathds{1}_{z=1}\, q((0,+\infty)\times \lbrace 0 \rbrace) \delta_{(x + \mathbf{T},z)} \smallskip\\
     +\mathds{1}_{z=1}\, q((0,+\infty)\times \lbrace 1 \rbrace) \frac{1}{2}\delta_{(x - 2\mathbf{P},z)}
     \end{array}\right].
     \]
     Hence $\Vert J(q,x) \Vert \leq \frac{7}{4} \frac{v}{2m^2}$, and for $q,q' \in \mathcal{P}(\N)$ we have: 
     \[
         \Vert J(q,x) -J(q',x) \Vert \leq \frac{7}{4} \frac{v}{2m^2} \Vert q - q' \Vert      
     \]
     
     Thus by Proposition 2.3 of Graham \cite{graham2000chaoticity} (see Proposition \ref{2.3 Graham}), there is a unique solution to the non-linear martingale problem defined in Lemma \ref{Non linear Martingal problem MF}.\\
     Hence, $\delta_\mu$ is the unique limit point of $\Pi^N$, and by Lemma \ref{2.2 Sznitman}, we have that $(Y^N)_N$ is $\mu$-chaotic.
     \qed

\subsection{Proof of Corollary \ref{Corol Evolution Equation} }
Since, $\forall t >0$ $M^f(t)$ is a $\mu$-martingale, taking for all $y \in \R_+^*$ $f = \mathds{1}_{(y,0)}$ (resp $f = \mathds{1}_{(y,1)}$) the indicator function in $(y,0) \in \R_+ \times \lbrace 0,1 \rbrace$ (resp in $(y,1) \in \R_+ \times \lbrace 0,1 \rbrace$) we get: for all $t>0$
\[
0 = \int_{D(\R_+,\N)} \left(\mathds{1}_{\substack{X(t)=y \\ Z(t) = 0}} - \mathds{1}_{\substack{X(0)=y \\ Z(t) = 0}} - \int_0^t L_{\mu_u} \mathds{1}_{(y,0)} (X(t),Z(t)) \text{d}u\right)\mu(\text{d}X)
\]
So with $(Y(t),z)_t$ a process of law $\mu$ we have: for all $y\in \R_+$

\begin{eqnarray}\label{Preuve Coroll Eqn 1}
\PP(Y(t)=y,z=0) - \PP(Y(0)=y,z=0) = \int_0^t \E (L_{\mu_u} f(Y(u),0)) \text{d}u
\end{eqnarray}

Let us denote $(Y_n)_n$ the Markov chain induced by $(Y(t))_t$ and $(N_t)_t$ the Poisson process counting the number of jumps the process do. Since we have:
\[
\begin{array}{rcl}
\PP(Y(t) >u,z=0) &=& \sum_{n=0}^{+\infty} \PP(N_t = n) \PP(Y_n > u,z=0)\medskip\\ & = & \sum_{n=0}^{+\infty} e^{-vt/m^2}\frac{1}{n!} \left(\frac{vt}{m^2}\right)^n  \PP(Y_n >u,z=0).
\end{array}
\] Hence we have that for all $u\in \R$, $t \mapsto \PP(Y(t)>u)$ is continuous, deriving equation (\ref{Preuve Coroll Eqn 1}) and doing the same reasoning for $(y,1)$ we get the evolution equations (\ref{Evol B}),(\ref{Evol R}).
\qed
\section{Extension to other models}
The two previous theorems (Theorem \ref{Thm Spatial model to Intermediate model} and Theorem \ref{Thm proposition chaos Modèle inter to MF}) can be applied not only to the Demographic prisoner's dilemma but also on a more general class of models. Let us introduce those models and then state the theorems in those general context. 
\subsection{Model}
As in Section \ref{Model}, we keep the spatial structure but we consider more general games and so more strategies (but which are always fixed).\\

\textbf{Model}
\begin{enumerate}
	\setlength\itemsep{0.2em}
\item Let $\mathfrak{G} = (\mathfrak{V}, \mathfrak{E})$ be a fixed simple graph connected.
\item $N$ particles are randomly placed on $\mathfrak{V}$ ($N \in \N^*$). For a particle $i \in \lbrace 1,\dots, N \rbrace$, let us denote $X_i(t)$ its position at time $t \in \R_+$. 
\item The $N$ particles move according to continuous time independent symmetric simple random walks on the graph $\mathfrak{G}$. Thus we associate to each particle $i$ a Poisson process with parameter $d\in \N^*$. When the Poisson process realizes, $X_i(t)$ jumps to one of its nearest neighbors (given by $\mathfrak{E}$) with equal probability. Then two (or more) particles can be on the same site at the same time.
We call $E_1:= \left(\mathfrak{V}\right)^N$ the discrete compact state space of the Markov process $(X(t))_t:= (X_1(t),\dots,X_N(t))_t$.

\item Each particle $i \in \lbrace 1,\dots N \rbrace$ has a wealth at time $t \geq 0$ denoted $Y_i(t)\in \R$.
\item Let $a\in \N^*$, given a matrix $G=(G(b,b'))_{b,b'} \in M_a(\R)$, for all $i \in \lbrace 1,\dots,N\rbrace$, $Y_i$ evolves following the symmetric game of payoff matrix $(G,G^T)$ (with actions $A_1,\dots, A_a$).\\
We make the wealth evolve the same way we do in Section \ref{Model}, that is only particles on the same site can play together and the wealth evolves accumulating the payoffs of the games.
\item The strategies of the individuals are fixed.  
\begin{itemize}
	\item Let $L \in \N^*$ be the number of fixed strategies (possibly mixed).
	\item We define for all $\ell \in \lbrace 1,\dots L \rbrace$ the strategy $\theta_\ell$ such that a player of strategy $\theta_\ell$ plays action $A_b$ ($b \in \lbrace 1,\dots a  ,\rbrace$) with probability $\alpha_b^\ell \in [0,1]$ (with $\sum_{b=1}^{a} \alpha_b^\ell = 1$).
	\item Each individual $i$ has a fixed strategy coded in a parameter $z_i \in \lbrace \theta_1,\dots, \theta_L \rbrace$. 
\end{itemize} 
We call $E_2:= (\R \times \lbrace \theta_1, \dots, \theta_L \rbrace)^N$ the state space of the stochastic process $ (Y(t),z)_t:= ((Y_1(t),z_1),\dots,(Y_N(t),z_N))_t$. 
\item Each strategy $\theta_\ell$ has a death domain $D_\ell$ such that: if the wealth $Y_i$ of an individual $i$ of strategy $\theta_\ell$ reaches $D_\ell$ then the particle dies and cannot play with other particles. For example in Section 2, $D_0=D_1=]-\infty,0]$.
\item We also make the following assumption: particles are indistinguishable, \emph{i.e.} the distribution of $(X,{Y},z)$ is exchangeable.
\end{enumerate}
The Markov infinitesimal generator of this model is: for all $f \in \overline{C}(E)$

\begin{eqnarray}\label{Gene extension}
\mathcal{A}f(x,y,z) = \mathcal{A}_d f(x,y,z) + \mathcal{A}_g f (x,y,z)
\end{eqnarray}

where \begin{eqnarray} \label{Gene dep extension}
\mathcal{A}_d f(x,y,z) = \sum\limits_{i=1}^{N} \sum\limits_{c=1}^2 \sum\limits_{\epsilon = \pm 1} \frac{d}{2 \times 2} [f(x + \epsilon e_i^c,y,z) - f(x,y,z)].
\end{eqnarray}

and for all $f \in \overline{C}(E)$: 
\begin{eqnarray}\label{Gene Jeu extension}
\mathcal{A}_g f(x,y,z)= \sum\limits_{\substack{(i,j)\in \lbrace 1,\dots,N \rbrace^2 \\ i \neq j}} \frac{1}{2}\mathds{1}_{x_i = x_j} g^f_{i,j}(x,y,z)
\end{eqnarray}  where $g^f_{i,j} (x,y,z)$ is equal to
\[
\sum_{(k,l) \in \lbrace 1,\dots, L \rbrace^2} \mathds{1}_{\substack{z_i = k\\Y_i \notin D_k}} \mathds{1}_{\substack{z_j = \ell\\Y_j \notin D_\ell}}. \sum_{\substack{b=1\\b'=1}}^{a} \alpha_b^k \alpha_{b'}^\ell [f(x,y + G(b,b')e_i + G(b',b)e_j,z) - f(x,y,z)].
\]
For all $f$ continuous bounded functions from $E_2$ to $\R$ and $(y,z) \in E_2$: we denote by $g^f_{i,j} (y,z)$ the following quantity:

\[
\sum_{(k,l) \in \lbrace 1,\dots, L \rbrace^2} \mathds{1}_{\substack{z_i = k\\Y_i \notin D_k}} \mathds{1}_{\substack{z_j = \ell\\Y_j \notin D_\ell}}. \sum_{\substack{b=1\\b'=1}}^{a} \alpha_b^k \alpha_{b'}^\ell [f(y + G(b,b')e_i + G(b',b)e_j,z) - f(y,z)].
\]

\subsection{Theorems}
Let us begin with the homogenization theorem. 
Let us denote $\pi$ the unique invariant distribution of a random walk on $\mathfrak{G}$. Let us denote $\mathfrak{m} = \sum\limits_{x \in \mathfrak{V}} \pi^2(x)$ the probability that particle $i$ and $j$ are on the same site ($i,j \in \lbrace 1,\dots N \rbrace$). 
\begin{theorem}
	Let $(X^d(t),Y^d(t),z)_t$ be a stochastic process with initial distribution $\nu \in \mathcal{P}(E_1 \times E_2)$ with second marginal denoted $\nu_y$ and generator $\mathcal{A}$ defined in $(\ref{Gene extension}),\ (\ref{Gene dep extension})$, $(\ref{Gene Jeu extension})$\\
				
		Then $(Y^d,z)_d$ converges in distribution (as a family of stochastic processes) to the Markov process $(\overline{Y},z)$ with initial distribution $\nu_y$, infinitesimal generator $\overline{\mathcal{A}}$ and domain $\overline{C}(E_2)$: for all $f\in \overline{C}(E_2)$
		\[
		\overline{\mathcal{A}}f(y,z) = \frac{v}{2} \mathfrak{m} \sum\limits_{\substack{(i,j) \in \lbrace 1,\dots N \rbrace^2\\
				i \neq j}}  g^f_{i,j} (y,z).
		\]
		
		\end{theorem}

As in Theorem \ref{Thm Spatial model to Intermediate model}, this theorem shows that when $d\to +\infty$, the wealth in the spatial model behaves as if the model is a random matching system.
\begin{proof}
	The proof is similar to the proof of Theorem \ref{Thm Spatial model to Intermediate model}. The only differences are: 
	\begin{itemize}
		\item In (A1), $M$ is now $\max\limits_{b,b'\in \lbrace1,\dots a \rbrace} G((b,b'))$.
		\item In (A4), we have that for all $f \in \overline{C}(E)$ constant on the first variable:\[
		\sup\limits_{d\in \N^*} \E\left(\int\limits_0^T|\mathcal{A}f(X^d(t),Y^d(t),z)|^p\,\text{d}s \right) \leq (N^2 vL^2 a^2 \Vert f \Vert_\infty)^pT <+\infty
		\]
		\item In the conclusion now the probability that two particles are on the same site is $\mathfrak{m}$ instead of $\frac{1}{m^2}$.
	\end{itemize}
	The rest of the proof is the same as in Theorem \ref{Thm Spatial model to Intermediate model}.

\end{proof}
\vspace{3pt}
	
The next theorem is the propagation of chaos theorem in the extended context. We also will slow the time considering the following generator with domain $\overline{C}(E_2)$
\begin{eqnarray}\label{Gene temps ralenti random matching extension}
\overline{\overline{\mathcal{A}}}f(y,z) = \frac{v \mathfrak{m}}{2}\frac{1}{N} \sum\limits_{\substack{(i,j) \in (\lbrace 1,\dots N \rbrace)^2\\
		i \neq j}}  g^f_{i,j} (y,z).
\end{eqnarray}
Let us state the theorem.
\begin{theorem}
		Let $(\hat{Y}^N(t),z)_t$ be a stochastic process (relative to a filtration $(\mathcal{F}_t)_t$) with generator $\overline{\overline{\mathcal{A}}}$ and initial distribution $\nu ^{\otimes N}$.\\

	Then there exists a probability measure $\mu$ on ${D}(\R_+,\R\times \lbrace \theta_1,\dots, \theta_L \rbrace)$ such that:\\
	the sequence of probability distributions $(\mathcal{L}(\hat{Y}^N,z^N))_N$ is $\mu$-chaotic.\\
		The distribution $\mu \in \mathcal{P}(D(\R_+,\R \times \lbrace \theta_1,\dots,\theta_L))$ has non-linear generators $(L_{\mu_t})_t$ with domain $\overline{C}(\R \times \lbrace \theta_1,\dots, \theta_L \rbrace)$ and initial distribution $\nu$.\\
	The generator $(L_{\mu_t})_t$ is given by:
	\begin{eqnarray}
	L_{\mu_t} f(y,z) =  \frac{1}{2} \sum_{\ell =1}^{L}\mathds{1}_{\substack{z = \theta_\ell\\
		y \notin D_\ell}} \sum_{k=1}^{L} \mu_t(D_k^c \times \lbrace \theta_k \rbrace) \sum_{\substack{b=1\\
		b'=1}}^{a} \alpha_b^k \alpha_{b'}^\ell [f(y+G(b,b'),z) - f(y,z)]
	\end{eqnarray}
\end{theorem}

\begin{corollary}
	The previous $\mu$ can be described by the following evolution equations.\\
	Let $(Y(t),z)_t$ be a process of law $\mu$ starting with the initial distribution $\nu$.\\
	We have: for all $\ell \in \lbrace 1,\dots  L\rbrace$ and all $y \notin D_\ell$ $\dro \PP(Y(t) = y,z = \theta_\ell)$ is equal to
	\[
     \begin{array}{l}
		 \sum\limits_{k = 1}^{L} \PP(Y(t) \notin D_k, z = \theta_k) \sum\limits_{\substack{b = 1\\b'=1}}^{a} \alpha_{b}^k \alpha_{b'}^\ell \PP(Y(t) = y - G(b,b'),z = \theta_\ell)\mathds{1}_{(y - G(b,b')) \notin D_\ell} \medskip\\
		  - \PP(Y(t) = y, z = \theta_\ell)\left(\sum\limits_{k=1}^{L} \PP(Y(t) \notin D_k, z=\theta_k)\right)
     \end{array}
\]
\end{corollary}

\begin{proof}
	The proves of the theorem and corollary are exactly the same as those of Theorem \ref{Thm proposition chaos Modèle inter to MF} and Corollary \ref{Corol Evolution Equation}.
\end{proof}

\section*{Acknowledgment}
I would like to thank my advisors for the review Laurent Miclo and J\'er\^ome Renault. Thanks also to Pierre Jeambrun for his review of the English. Thanks to Xavier Bressaud for initiating me into the world of Research.

\bibliographystyle{plain}
\bibliography{ArticleArxiv3}

\end{document}